\documentclass[11pt,fleqn]{amsart}

\usepackage[dvipdfm]{graphicx}
\usepackage{epsfig}
\usepackage{amsmath}
\usepackage{amssymb}
\usepackage{amscd}
\usepackage{latexsym}
\usepackage{tabularx}
\usepackage{a4wide}
\usepackage[usenames]{color}
\usepackage{enumerate}
\usepackage{subfigure}
\usepackage{url}
\usepackage{verbatim}

\usepackage[matrix,arrow,tips,curve]{xy}
\usepackage{pb-diagram,pb-xy}
\usepackage{verbatim}
\usepackage{mathrsfs}
\usepackage{color}
\usepackage[leqno]{amsmath}

\newcommand{\ta}{\theta}

\small\normalsize

\newcommand{\T}{\mathrm{T}}

\newtheorem{thm}{Theorem}[section]

\newtheorem{lemma}[thm]{Lemma}

\newtheorem{prop}[thm]{Proposition}

\newtheorem{claim}[thm]{Claim}
\newtheorem{cor}[thm]{Corollary}

\newtheorem{remark}[thm]{Remark}

\newcommand{\ZZ}{\mathbb{Z}}
\newcommand{\R}{\mathbb{R}}
\newcommand{\RR}{\mathbb{R}}

\newcommand{\C}{\mathfrak{C}}

\newcommand{\ord}{\mathrm{ord}}

\newcommand\fX{\mathfrak X}

\renewcommand{\div}{\mathrm{div}}

\renewcommand{\k}{\kappa}
\newcommand{\K}{\mathbb K}

\newcommand{\Spec}{\operatorname{Spec}}
\newcommand{\g}{\mathfrak g}

\newcommand{\an}{\operatorname{an}}
\newcommand{\ad}{\operatorname{ad}}
\newcommand{\onto}{\twoheadrightarrow}  
\newcommand{\into}{\hookrightarrow}

\renewcommand{\H}{\mathcal H}

\newcommand{\B}{\mathbb B}

  \newenvironment{dedication}
        {\vspace{.2ex}\begin{quotation}\begin{center}\begin{em}}
        {\par\end{em}\end{center}\end{quotation}}
\DeclareMathOperator\val{val}

\title{Equidistribution of Weierstrass points on curves over non-Archimedean fields}
\author{Omid Amini}
\address{CNRS - D\'epartement de math\'ematiques et applications, \'Ecole normale sup\'erieure, Paris}
\email{oamini@math.ens.fr}

\begin{document}
\maketitle

\begin{dedication}
D\'edi\'e \`a Jean Lannes
\end{dedication}
\begin{abstract}
 We prove equidistribution of Weierstrass points on Berkovich curves. 
 Let $X$ be a smooth proper curve of positive genus over a complete algebraically closed 
 non-Archimedean field $\K$ of equal characteristic zero with a non-trivial valuation. 
 Let $L$ be a line bundle of positive degree on $X$. The Weierstrass points of powers of 
 $L$ are equidistributed according to the Zhang-Arakelov measure on $X^{\an}$. This provides a non-Archimedean analogue of a theorem of Mumford and Neeman. 
 
 Along the way we provide a description of the reduction of Weierstrass points, answering a question of Eisenbud and Harris.
\end{abstract}

\section{Introduction}
 
A theorem of Mumford and Neeman~\cite{Mumford, Neeman} (see also~\cite{Olsen}) 
states that for a compact Riemann surface 
$S$ of positive genus, and for a line bundle 
$L$ of positive degree  on $S$, the discrete measures $\mu_n$ supported on
Weierstrass points of 
$L^{\otimes n}$ converge weakly to the Arakelov-Bergman measure on $S$ when $n$ goes to infinity.  
Our aim in this paper is the prove a non-Archimedean analogue of Mumford-Neeman theorem.

Let $\K$ be 
a complete algebraically closed valued 
non-Archimedean field with valuation ring $R$ and residue field $\k$. 
We suppose that the valuation of $\K$ is non-trivial. In addtion, we 
assume that both $\K$ and $\k$ are of 
characteristic zero.

 Let $X$ be a smooth proper curve of positive genus over $\K$. 
 The analogue of the Arakelov measure in non-Archimedean Arakelov theory 
 is the canonical admissible measure constructed by 
Zhang~\cite{Zhang}. Fixing a skeleton $\Gamma$ of $X^{\an}$, which is a finite metric graph, 
the canonical admissible measure $\mu_{\mathrm{ad}}$ is a measure of total mass one with support in $\Gamma$. 
The skeleton $\Gamma$ comes with an inclusion map $\iota: \Gamma \hookrightarrow X^{\an}$, 
which allows to pushforward  
$\mu_{\ad}$ to $X^{\an}$. By the explicit form of $\mu_{\ad}$, and since $\mu_{\ad}$ does not have support on bridge edges,
it is easy to see that 
the resulting measure $\iota_*(\mu_{\ad})$ on $X^{\an}$ does not depend on the choice of the 
skeleton  $\Gamma$. This leads to a well-defined measure $\mu_{\ad}$ on the 
Berkovich analytification $X^{\an}$ of $X$, supported on the minimal skeleton of $X^{\an}$. We will review the definition in Section~\ref{sec:zhang}.

Let now $L$ be a line bundle of positive degree on $X$, and denote by 
$\mathcal W_n =\sum_{x\in X(\K)} w_{n,x} (x)$ the 
Weierstrass divisor of $L^{\otimes n}$, that we view as a divisor with support in 
$X(\K) \subset X^{\an}$. Consider the discrete measure 
$$\frac 1{\deg(\mathcal 
W_n)}\delta_{\mathcal W_n} = \frac 1{\deg(\mathcal 
W_n)} \sum_{x\in X(\K) \subset X^{\an}} w_{n,x} \delta_x$$
supported on a finite set of points of type I in $X^{\an}$. We have

\begin{thm}\label{thm:main2} Notations as above, the 
measures $\frac 1{\deg(\mathcal 
W_n)}\delta_{\mathcal W_n}$ converge  weakly to the canonical admissible measure
$\mu_{\mathrm{ad}}$ on $X^{\an}$. 
\end{thm}

Since continuous functions of the form $\tau^*(f)$, 
for the retraction $\tau: X^{\an} \onto \Gamma$ of $X^{\an}$ to a skeleton and 
continuous function $f$ 
on $\Gamma$, are dense in the space of continuous functions on $X^{\an}$, 
in order to  prove Theorem~\ref{thm:main2}, it will be enough to fix a skeleton 
$\Gamma$ of $X^{\an}$ and prove the following equidistribution theorem. 
Denote by $W_n  = \tau_*(\mathcal W_n)$ the  
reduction of $\mathcal W_n$ on $\Gamma$. Define the discrete measure
 $\mu_n := \frac 1{\deg(W_n)}\delta_{W_n}$.

\begin{thm}\label{thm:main} Notations as above, 
the measures $\mu_n$ converge weakly to the canonical admissible measure $\mu_{\ad}$ on $\Gamma$.
 \end{thm} 
 
 A consequence of our results, which seems to be also new,
 is the following.
 \begin{cor}\label{cor:onecomponent}
   Let $X$ be a smooth proper curve of genus $g>0$ over a discrete valuation field $K$ of 
   equicharacteristic zero. 
   Let $\mathfrak X$ be a regular semistable model of $X$ over $R'$, 
   the valuation ring of a finite extension $K'$ of $K$. Let $L$ be a line bundle 
   of positive degree on 
   $X$ and  denote by $\mathcal W_n$ (by an abuse of the notation) the (multi)set of
   Weierstrass points of $L^{\otimes n}$.
   Let $\C_0$ be an irreducible component of the special fiber $\mathfrak X_s$ of genus $g_0$, 
   and denote by $\mathcal W_{0,n}$ the (multi)set of Weierstrass points whose 
   specialization lie on $\C_0$. Then 
   $|\mathcal W_{0,n}|/|\mathcal W_n|$ tends to $g_0/g$.   
 \end{cor}

\begin{remark}\rm A combination of our results and Mumford-Neeman theorem shows the following. 
Notations as in the above corollary, suppose in addition that the residue field 
$\k\simeq \mathbb C$, and $g_0\geq 2$. Then the reduction of $\mathcal W_{0,n}$ in $\C_0$
is equidistributed 
according to the Arakelov 
measure on $\C_0^{\an}$. A similar result holds if we assume that $\k$ is an algebraically 
closed complete 
non-Archimedean field, with a residue field of characteristic zero; 
the equidistribution is then with respect to the Zhang-Arakelov measure on 
$\C_0^{\an}$. We omit the details.
\end{remark}

The rest of this paper is devoted to the proof of Theorem~\ref{thm:main}. 

\medskip

A useful machinery in the proof 
is the framework of limit linear series  introduced in a joint paper with M. Baker~\cite{AB}, and further developed 
in an upcoming paper of the author in~\cite{A}. 
The limit linear series from~\cite{AB} provided an extension  of crude limit linear series in the terminology of Eisenbud-Harris~\cite{EH86} 
to any semistable 
curve, while the limit linear series that we consider here correspond to extensions 
to any semistable curve of the 
refined limit linear series in that terminology~\cite{EH86}. We will recall the basic set-up in 
Section~\ref{sec:lls}, and refer to~\cite{A} for more details and results.

\medskip

In Section~\ref{sec:reduction}, we consider the reduction of Weierstrass divisors on skeleta, 
and show that limit linear series allow to provide a complete description of the points in the support as well 
as their coefficients in the case $\k$ has characteristic zero, c.f.~Theorem~\ref{thm:SpW} below. 
This gives a complete and quite simple 
answer to a question of Eisenbud-Harris~\cite{EH87} on the reduction of Weierstrass points on a 
(semistable) degenerating family of curves over complex numbers, furthermore, refining the 
previous partial result of Esteves-Medeiros~\cite{EM}. 
\begin{thm}\label{thm:SpW}
Let $L = \mathcal O(\mathcal D)$ be a line bundle of degree $d$ and rank $r$, 
 and denote by $\mathcal W$ its 
 Weierstrass divisor on $X$. Let $\Gamma$ be a skeleton of $X^{\mathrm an}$  
 and let  $(D, \mathfrak S)$ be the limit $\g^r_d$ induced on $\Gamma$ by 
 the specialization theorem. Write $D = \sum_{x\in\Gamma} d_x(x)$.
 We have 
 \[\tau_*(\mathcal W) = \sum_{x\in \Gamma} c_x (x),\]
 where the coefficient $c_x$ of a point $x$ has the following expression:
 \[c_x = (r+1)d_x + \frac{r(r+1)}{2}(2g_x-2+\val(x)) - \sum_{\nu\in \T_x(\Gamma)} \sum_{i=0}^r s_i^\nu.\]
 \end{thm}
In the above theorem $g_x$ and $\val(x)$ denote the genus and the valence of $x$ in $\Gamma$, respectively,
$\T_x(\Gamma)$ denotes the set of unit tangent directions to $\Gamma$ 
at $x$, and $s_i^\nu$ are the integers underlying the definition of
the limit linear series $(D, \mathfrak S)$ on $\Gamma$, which correspond to the slopes along $\nu$ 
of the reduction of rational functions in the complete linear series defined by 
$\mathcal O(\mathcal D)$, c.f. Section~\ref{sec:lls}.

The fact that such a simple description of the reduction of 
 Weierstrass divisors in terms of a  
 combinatorial data exists  is probably a specific property of the residue field 
 $\k$ being of characteristic zero. 
 This is what refrains us from stating the 
 equidistribution theorem in its full generality.  However, 
 in the case $\k$ has positive characteristic, we still have a similar description 
 of the reduction of the Weierstrass divisor in terms of the reduction of the Wronskian at points of type II
 in $\Gamma$, see Theorem~\ref{thm:SpW2}. 
 We expect indeed that the equidistribution theorem remains 
 valid in any 
 characteristic; we discuss in Remark~\ref{rem:positive} what has to 
 be done in order to get such a result from Theorem~\ref{thm:SpW2}.

\medskip

A second tool essential in the proof is the theory of Okounkov bodies (in dimension one), 
and a well-known equidistribution phenomena for Okounkov bodies~\cite{KK, LM, Bouk}. Once 
Theorem~\ref{thm:SpW} has been discovered, the 
appearance of the Okounkov bodies becomes indeed quite natural if we remember that to any 
type II point $x$ of 
$X^{\an}$ is associated a curve $\C_x$ over $\k$, whose points are in bijection with the branches of 
$X^{\an}$ adjacent to $x$. Fixing such a branch $\nu\in \T_x(X^{\an})$ of $X^{\an}$, and 
considering the slopes of 
rational functions along $\nu$ in the complete linear series defined by powers of 
$L$ allows to define an interval 
$\Lambda^\nu$ of volume $d = \deg(L)$ in $\mathbb R$. The local equidistribution theorem then asserts 
that the normalized slopes are 
equidistributed in $\Lambda^\nu$, c.f. Theorem~\ref{thm:equilocal}.

\medskip

Theorem~\ref{thm:SpW} and~Theorem~\ref{thm:equilocal}, 
along with basic properties of limit linear series from Section~\ref{sec:lls}, 
and a careful analysis of the variation of the minimum slope along edges of 
$\Gamma$,  then allow to finish the proof of Theorem~\ref{thm:main}.

 \medskip

We like to mention that the question of existence of a non-Archimedean 
version of Mumford-Neeman theorem has been 
implicitly asked in some places in the literature~\cite{Burnol, deJong}. 
We first came to learn about this through discussions with Matt Baker on the behavior of 
(divisorial) Weierstrass points on a metric graph $\Gamma$, which, we recall, are 
 intrinsically defined in terms of the divisor theory on $\Gamma$ 
(and without any reference to a choice of an algebraic curve $X$ with reduction graph $\Gamma$). 
While this general question still remains open, we note that Theorem~\ref{thm:SpW} 
predicts an intrinsic integer number can be assigned to any 
connected component of the locus of the (divisorial) Weierstrass points in $\Gamma$, 
thus reducing Baker's question to the question of understanding these numbers.

We would like to use this opportunity to thank Matt Baker for sharing his  
questions and conjectures, and for his  continuous support and collaboration, on related topics. We thank Matt Baker and 
Farbod Shokrieh for helpful comments on a preliminary draft of this paper.

 \medskip
 
 In the rest of this introduction, we will fix the notation, and briefly recall some well-known
 results which will be used in the next sections.

\subsection{Weierstrass divisor}

Let $\K$ be an algebraically closed  field, $X$ a smooth proper curve over $\K$, and $L$ a line bundle of positive degree on $X$. 
Consider a basis  
$f_0, \dots, f_r$ be a basis for $H^0(X,L)$, and set $\mathcal F = \{f_0,\dots, f_r\}$. 

If the characteristic of $\K$ is zero, then we consider the Wronskian 
$\mathrm{Wr}_\mathcal F$ which is a global section of 
$\Gamma(X, L^{\otimes (r+1)}\Omega^{\otimes \frac{r(r+1)}2})$. 
Intrinsically, it is defined as follows: consider the jet bundle 
$J_r = \pi_{1*}\Bigl(\,\pi_{2}^*(L)/I^{r+1}\,\Bigr)$, where 
$I$ is the ideal of the diagonal in $X\times X$, and $\pi_1$ and $\pi_2$ are the projections into 
the first and the second factor. There is a natural induced filtration on $J_r$ by powers of 
$I$, and the quotients are identified with $L \otimes \Omega_X^{\otimes i}$. 
The set of global sections $f_0,\dots, f_r$ define sections $\pi_2^*(f_i)$ of $J_r$, 
whose determinant will be a global section of 
$L^{\otimes (r+1)}\otimes \Omega^{\otimes r(r+1)/2}$ 

In local coordinate, for any point $p\in X(\K)$, 
we have $L_p \simeq \mathcal O_{X,p}$ as an $\mathcal O_p$-module. 
Taking a generator $g_p$ of $L_p$, any global section $f_i$ can be written as 
$f_i = f_{i,p}g_p$ for $f_{i,p} \in \mathcal O_p$. The local ring $\mathcal O_p$ 
is a discrete valuation ring with a uniformizer $t_p$. 
We can thus define $f_{i,p}^{(j)}$ recursively by 
$f_{i,p}^{(j)} = \frac d{d\tau_p}f_{i}^{(j-1)}$ for any $j\geq 0$, with 
$f_{i,p}^{(0)}=f_{i,p}$. The stalk of the Wronskian $\mathrm{Wr}_{\mathcal F}$ at $p$ is given by 
$$ \mathrm{Wr}_{\mathcal F, p} = \det \Bigl(f_{i,p}^{(j)}\Bigr)_{0\leq i,j\leq r} g_p^{(r+1)} 
(dt_p)^{\frac {r(r+1)}2}\in H^0(\mathcal O_p, 
L_p^{\otimes (r+1)} \otimes \Omega_p^{\otimes \frac {r(r+1)}2}). $$
In particular, it does not depend on the choice of $g_p$ and 
the local uniformizer $t_p$.

\medskip

In positive characteristic $\mathrm{char}(\K)>0$, it is well-known that Weierstrass points 
are more subtle, and the definition given above has to be modified, see e.g.~\cite{Neeman2}. 
 We should instead consider the Hasse derivative with respect to a separating parameter $t$ for 
$\K(X)$~\cite{Schmidt, SV}. 
There exists $t \in \K(X)$ such that $\K(X) / \K(t)$ is a finite separable extension. 
The Hasse derivatives 
$D^{(j)}_t$, defined on $\K[t]$ by $D^{(j)}_t(t^{m}) = \binom{m}{j}t^{m-j}$, uniquely extend to 
$\K(X)$. There exist a set of $(r+1)$ integers $0\leq b_0 <  b_1<\dots < b_r$ such that 
$\det\Bigl(\,D_t^{(b_j)}f_i\,\Bigr)_{0\leq i,j\leq r} \neq 0$. In addition, $(b_0,\dots,b_r)$ is chosen to be minimum in the lexicographic order with respect to this property. 
The corresponding Wronskian $\mathrm{Wr}_{\mathcal F}$ defined by $\mathrm{Wr}_{\mathcal F} = 
\det\Bigl(\,D^{(b_j)}f_i\,\Bigr)_{i,j=0}^r$
defines a well-defined global section 
of $L^{\otimes (r+1)}\otimes \Omega^{\otimes \sum_{i=0}^r b_i}$, 
whose zero divisor is independent of the choice of $t$, and by definition, 
is called the Weierstrass divisor of $L$~\cite{Lak81, Lak84, Schmidt, 
SV}. Note that the Wronskian admits a similar description in terms of jumps of jet bundles 
in positive characteristic~\cite{Lak81, Lak84}.

\subsection{Berkovich analytic curves}
We provide a brief discussion of the structure of Berkovich analytic 
curves, and introduce the main notation which will be used later. 
For further details, we refer to~\cite{BPR, Ber90, Ducros2}.

\medskip

Let $X/\K$ be an algebraic variety. The topological space underlying the Berkovich analytification $X^{\an}$ of $X$ is described as 
follows. Each point $x$ of $X^{\an}$ corresponds to a scheme-theoretic point $X$, 
with residue field 
$\K(x)$, and an extension $| \, |_x$ of the absolute
value on $\K$ to $\K(x)$.   
The topology on $X^{\an}$ is the weakest one for which $U^{\an} \subset X^{\an}$ is open for 
every open affine
subset $U \subset X$ and the function $x \mapsto |f|_x$ 
is continuous for every $f \in {\mathcal O}_X(U)$.  By definition, the set $X(\K)$ of closed points of $X$ is naturally included in
$X^{\an}$,  and has a dense image. 
The space $X^{\an}$ is locally compact, Hausdorff, and locally path-connected.
Furthermore, $X^{\an}$ is compact iff $X$ is proper, and path-connected iff $X$ is connected.
Analytifications of algebraic varieties is a subcategory of a larger category of 
$\K$-analytic spaces, and e.g., open subsets of $X^{\an}$ come with 
a $\K$-analytic structure in a natural way~\cite{Ber90}.

For any point $x$ of $X^{\an}$, 
 the completion of the residue field 
$\K(x)$ of $X$ with respect to $|\,|_x$ is denote by $\H(x)$, and the residue field of the valuation field $(\H(x), |\,|_x)$ is denoted by
$\widetilde{\H(x)}$.

\subsubsection{Structure of analytic curves}
For an analytic curve $X^{\an}$, the points can be classified into four types. 
By Abhyankar's inequality, 
$\textrm{tr-deg}\Bigl(\widetilde{\mathcal H(x)}/\k\Bigr)+\textrm{rank}\Bigl(|{\mathcal H}(x)^\times| / 
|\K^\times|\Bigr) \leq 1$, where the rank is that of a finitely generated abelian group.
 The point $x$ is then of type I if it belongs to $X(\K)$ in which case, 
${\mathcal H}(x) \cong \K$, of type II 
if the transcendence degree of $\widetilde {\mathcal H(x)}/\k$ is one, of type III if 
the rank of the valuations extension is one, and of type IV otherwise.

\subsubsection*{Semistable vertex sets and skeleta}

A semistable vertex set for $X^{\an}$ is a finite set $V$ of points of $X^{\an}$ of type II 
such that $X^{\an} \setminus V$ is isomorphic to a
disjoint union of a finite number of open annuli and an infinite number of open balls.  By semistable reduction theorem, semistable vertex sets always exist,
and more generally, any finite set of points of  type II in $X^{\an}$ 
is contained in a semistable vertex set.
The skeleton $\Gamma = \Sigma(X,V)$ of $X^{\an}$ with respect to a semistable vertex set 
$V$ is the subset of $X^{\an}$ defined as the union of $V$ and the skeleton of
each of the open annuli in the semistable decomposition  associated to $V$.
Using the canonical metric on the skeleton of the open annuli, 
$\Gamma$ comes naturally equipped with the 
structure of a finite metric graph  contained in $X^{\an}$. In addition, $\Gamma$ has a 
natural model 
$G=(V,E)$ where the edges are in correspondence with the annuli in the semistable decomposition. In this paper, we only consider semistable vertex sets whose associated model is a simple graph, i.e., 
without loops and multiple edges.

Semistable vertex sets for $X^{\an}$ correspond bijectively to semistable formal 
models $\fX$ for $X$ over $R$.

\subsubsection*{Retraction to the skeleton} 
Let $\Gamma$ be a skeleton of $X^{\an}$ defined by a semistable vertex set $V$. 
There is a canonical retraction map $\tau : X^{\an} \onto \Gamma$ which 
is in fact a strong deformation retraction~\cite{Ber90}. In terms of the semistable decomposition, 
$\tau$ is identity on $\Gamma$, sends the points of each open ball $B$
to the unique point of $\Gamma$
 in the closure $\overline B$ of $B$, called the end of $B$, 
 and is the retraction to the skeleton for the open annuli~\cite{Ber90, BPR}.

\subsubsection*{Residue curves and the genus formula}
A point $x \in X^{\an}$ of type II has a (double) residue field $\widetilde{{\mathcal H}(x)}$ which is of transcendence degree one over $\k$.
We denote by $\C_x$ the unique smooth proper curve over $\kappa$ with function field 
$\widetilde{{\mathcal H}(x)}$, and denote by $g_x$ the genus of $\C_x$. 
If $V$ is any semistable vertex set for $X^{\an}$, then for any point of type II in $X^{\an}\setminus V$, 
$g_x=0$, and by semistable reduction theorem, we have the following genus formula:
\[
g= g(X) = g(\Gamma) + \sum_{x \in V} g_x,
\]
where $g(\Gamma) = |E|-|V|+1$, for $G=(V,E)$ the model of the skeleton $\Gamma = \Sigma(X,V)$, 
is the first Betti number of $\Gamma$. We extend the definition of $g(\cdot)$ to all points of 
$\Gamma$ by declaring $g(x)=0$ if $x$ is not a point of type II in $X^{\an}$, obtaining in this way an augmented metric graph in the terminology of~\cite{ABBR}.

\subsubsection*{Tangent vectors}

There is a canonical metric on $ {\mathbf H}(X^{\an})$ which restricts to the metric on 
$\Gamma = \Sigma(X^{\an},V)$ for any semistable vertex set $V$ for $X^{\an}$.  
Here, by $ {\mathbf H}(X^{\an})$, we denote $X^{\an}$ without the set of points of type I and IV.

A geodesic segment starting at $x \in X^{\an} \setminus X(\K)$ is an isometric 
embedding $\alpha : [0,\theta] \to X^{\an}\setminus X(\K)$ for some $\theta > 0$ 
such that $\alpha(0)=x$. Two geodesic segments starting at $x$ are called 
equivalent if they agree on a neighborhood of $0$. As usual, a 
tangent direction at a point $x$ is an equivalence class of geodesic segments 
starting at $x$. We denote by $\T_x = \T_x(X^{\an})$ the set of all tangent directions at $x$. 

For any simply connected neighborhood $U$ of $x \in X^{\an}$, there is a natural 
bijection between $\T_x$ and
the connected components of $U \setminus \{ x \}$.   
There is only one tangent direction at $x$ when $x$ is of type I; 
for $x$ of type III we have $|\T_x|=2$. (For $x$ of type IV we have $|T_x|=1$.)  
For a point $x$ of type II, there is a canonical bijection between $\T_x$ and 
$\C_x(\kappa)$, the set of closed points of the smooth proper curve $\C_x$ associated to $x$.  Points of $\C_x(\k)$ correspond to discrete valuations on 
$\widetilde{{\mathcal H}(x)}$ which are trivial on $\kappa$, and the resulting
bijection with $\T_x$ associates to a vector 
$\nu \in \T_x$, a discrete valuation 
${\rm ord}_{\nu} : \kappa(\C_x)^\times \to \ZZ$: If $x^\nu$ denotes 
the corresponding point of $\C_x(\k)$
then, for every nonzero rational function $\widetilde{f} \in \kappa(\C_x)$,
we have 
${\rm ord}_{\nu}(\widetilde{f}) = \ord_{x^\nu}(\widetilde{f})$.

\subsubsection{Reduction of rational functions and the slope formula}
\label{section:ReductionOfRationalFunctions}

Let $x \in X^{\an}$ be a point of type 2. For 
a nonzero rational function $f$ on $X$, there is an element 
$c \in \K^\times$ such that $|f|_x=|c|$. 
Define $\widetilde f \in \kappa(\C_x)^\times$ to be the image of $c^{-1}f$ in
$\widetilde{{\mathcal H}(x)} \cong \kappa(\C_x)$. Note that if the valuation of $\K$ has a 
section (which is the case for algebraically closed fields~\cite[Lemma 2.1.15]{MS}), this can be made well-defined; otherwise, it is well-defined up to a multiplicative scalar.

If $H$ is a $\K$-linear subspace of $\K(X)$, 
the collection of all possible reductions of nonzero elements of $H$, together with $\{ 0 \}$,
forms a $\k$-vector space $\widetilde H$. In addition, we have 
$ \dim \widetilde{H} = \dim H$ (c.f.~\cite{AB}).

A function $F:X^{\an}\to\RR$ is piecewise linear if 
for any geodesic segment $\alpha:[a,b]\into X^{\an}\setminus V$, the
pullback map $F\circ\alpha:[a,b]\to\R$ is piecewise linear.  
The outgoing slope of a piecewise linear function $F$ at a point
$x\in X^{\an}$ along a tangent direction $\nu\in \T_x$ is defined by 
\[ d_{\nu}F(x) = \lim_{t\to 0} (F\circ\alpha)'(t), \]
where $\alpha:[0,\theta]\into X^{\an}$ is a geodesic segment starting at $x$ which 
represents $\nu$.  
A piecewise linear function $F$ is called harmonic at a point $x\in X^{\an}\setminus V$ if
the outgoing slope $d_{\nu}F(x)$ is zero for all but finitely many $\nu \in \T_x$, and in addition 
$\sum_{\nu\in \T_x} d_{\nu}F(x) = 0$.

The following theorem will be essential all through the paper~\cite{BPR, BL, Thuillier}. 
It is called the slope formula in~\cite{BPR} and is also a consequence of the non-Archimedean 
Poincar\'e-Lelong formula~\cite{Thuillier}. 

\begin{thm}[Slope formula]  \label{thm:PL} Let $X$ be a smooth proper curve over $\K$, and 
 $f$ be a nonzero rational function in $\K(X)$. 
 Let $F = -\log |f|: X^{\an}\to\RR\cup\{\pm\infty\}$.
Let $V$ be a semistable vertex set of $X$ such that zeros and poles of $f$ 
are mapped to vertices in 
$V$ under the retraction map $\tau$ from $X^{\an}$ to the skeleton $\Gamma = \Sigma(X,V)$.  We have
  \begin{enumerate}
  \item[$(1)$]  $F$ is piecewise linear with integer slopes, and $F$ is linear on each 
  edge of $\Gamma \hookrightarrow X^{\an}$. 
  \item[$(2)$]  If $x$ is a type-$2$ point of $X^{\an}$ and $\nu\in \T_x$, then
    $d_{\nu} F(x) = \ord_\nu(\widetilde f_x)$.
  \item[$(3)$]  $F$ is harmonic at all $x \in {\mathbf H}(X^{\an})$.
  \item[$(4)$]  Let $x$ be a point in the support of $\div(f)$, let $e$ be the ray 
  in $X^{\an}$ with one endpoint $x$ and another endpoint $y\in V$, and
    let $\nu \in \T_y$ be the tangent direction represented by $e$.  Then $d_{\nu}
    F(y) = \ord_x(f)$.
  \end{enumerate}
\end{thm}

To each nonzero rational function $f$ on $X$ and each semistable vertex set $V$ for $X$, 
one associates a corresponding rational function $F = -\log|f|$ on the skeleton $\Gamma$.

As an application of Theorem~\ref{thm:PL}, we obtain the following~\cite{AB, baker}: 
For every nonzero rational function $f$ on $X$,
\[
\tau_*(\div(f)) =  \div(F).
\]

Here $\tau_*$ is the 
specialization map on divisors from curves to metric graphs induced from the retraction map 
$\tau$, 
and coincides with the reduction map in~\cite{CR, Zhang}.

 \subsection{Zhang's measure}\label{sec:zhang}
 
 Let $\Gamma$ be a metric graph with a simple
 graph model $G = (V,E)$. Each edge $e$ in $E$ has a length $\ell_e>0$.  
 For any point $x$ of $\Gamma$, we denote by $\T_x(\Gamma)$, 
the set of all outgoing unit tangent vectors to $\Gamma$ at $x$. The
directional derivative of $f$ at any point $x$ of $\Gamma$ 
along a tangent vector $\nu \in \T_x(\Gamma)$ is denoted by $d_{\nu}f( x)$

The space of piecewise smooth function on $\Gamma$ is denoted by $S(\Gamma)$. 
The metric graph $\Gamma$ has a natural Lebesgue measure denoted by $d\ta$. 
The Laplacian of $\Gamma$ is the (measure valued) operator 
$\Delta$ on $\Gamma$ which to a function $f\in S(\Gamma)$ associates the measure 

$$\Delta(f):= - f''d\ta -\sum_{p} \sigma_p \delta_p,$$
where $d\ta$ is the Lebesgue measure on $\Gamma$,  $\delta_p$ is the
Dirac measure at $p$, and $\sigma_p$ is the sum of the directional
derivatives of $f$ along tangent directions in $\T_p(\Gamma)$:
$$\sigma_p = \sum_{\nu\in \T_p(\Gamma)} d_{\nu}f( p ).$$

 Consider a measure $\mu$ on $\Gamma$ of total mass one. Fix a point $x\in \Gamma$ and
 consider the following Laplace equation
 \begin{equation}\label{eq:laplace}
 \Delta_y \, g(x,y) = \delta_x -\mu\,,
 \end{equation}
with uniformizing condition $\int_\Gamma g(x,y)\,d\mu(y) = 0$. It has a unique solution denoted by $g_{\mu}(x,y)$.
In addition, $g_\mu(x,y)$ has the following explicit representation in terms of the Green function $g_z(x,y)$ on $\Gamma$. Recall first that 
$g_z(x,y)$ is the unique solution to the equation 
\[\Delta_yg_z(x,y) = \delta_x - \delta_z,\]
under the condition $g_z(x,z)=0.$ Then we have $g_\mu(x,y) = \int_{\Gamma}g_z(x,y)d\mu(z)$. We refer to~\cite{BR} for more details.

\subsubsection{Canonical admissible measure $\mu_{\ad}$}
Let $D$ be a divisor and $\mu$ be a measure on $\Gamma$. We use the conventional notation that 
$g_\mu(D,y) = \sum_{p\in \Gamma} D(p)g_\mu(p,y)$.

The following theorem is proved by Zhang~\cite{Zhang}, and is a generalization of a theorem of 
Chinburg-Rumley~\cite{CR} for $D=0$ to any divisor of degree $\deg(D)\neq -2$.

\begin{thm}[Zhang~\cite{Zhang}]
Let $D$ be a divisor of degree different from $-2$ on a metric graph $\Gamma$ of positive genus. 
There is a measure $\mu_D$ and a constant $c_D$ 
such that for any point $x$ of $\Gamma$, one has 
\[c_D + g_{\mu_D}(D,x)+ g_\mu(x,x)=0.\]
In addition, the pair $(\mu_D, c_D)$ is  unique. 
\end{thm}

Let now $X$ be a smooth proper connected curve on $\K$, and $\Gamma$ 
a skeleton of $X^{\an}$ associated to a semistable vertex set $V$. 
The canonical divisor $K_X$ on $\Gamma$ is the divisor 
\[K_X:= \sum_{x\in \Gamma} \bigl(\,2g_x-2+\val(x)\bigr) (x).\]
Note that $\deg(K_X) = 2g(X)-2$. The metric graph $\Gamma$ is an augmented metric graph in the 
terminology of~\cite{ABBR}, and 
$K_X$ is equal to the sum  $K_\Gamma + 2K_g$, where $K_\Gamma$ is the canonical divisor of the
(unaugmented) metric graph 
$\Gamma$, and $K_g$ is the divisor associated to the genus function 
$g(\cdot)$ defined on the points of $\Gamma$.
The measure associated to the canonical divisor $K_{X}$ of the augmented metric graph 
$\Gamma$ is called the canonical admissible measure and is denoted by $\mu_{\ad}$. 
While not needed in the sequel, we note that the measure $\mu_0$ (associated to $D=0$) 
is the canonical measure of~\cite{CR}. 

The canonical admissible measure $\mu$ has the following explicit form~\cite{Zhang}. For any divisor $D$ we denote by $\delta_D$ the measure 
$\sum_{p\in \Gamma} D(p)\delta_p$ of total mass $\deg(D)$.
\begin{thm}[Zhang~\cite{Zhang}]
Notations as above, we have
\[\mu_{\ad} = \frac 1{g} \delta_{K_g} + \frac 1g \sum_{e\in E} \frac {1}{\ell_e+\rho_e}d\ta,\]
where $g$ is the genus of $X$, and for any edge $e$ in $E$ (the set of edges of the graph model of $\Gamma$), $\rho_e$ denotes 
the effective resistance between the two end points of $e$ in $\Gamma \setminus e$. 
\end{thm}
Here the metric graph $\Gamma \setminus e$ is defined by removing
the interior of the edge $e$  from $\Gamma$. 
The effective resistance $\rho_e$ in $\Gamma \setminus e$ 
between the two points $u$ and $v$ of the edge $e=\{u,v\}$ 
is formally defined as $\rho_e = g^e_u(v,v)$, where $g^e_z(z,y)$ is 
the Green function on $\Gamma \setminus e$, as described above. 

 Note that $\rho_e = \infty$ if $\Gamma \setminus e$ is disconnected, and is finite otherwise. 
 In particular, $\mu_{\ad}$ does not have support on the interior of any bridge edge of $\Gamma$ (in that case, 
 $\frac{\ell_e}{\ell_e+\rho_e} =0$). This shows that 
 the measure $\iota_*(\mu_{\ad})$ 
 is a well-defined measure on 
 $X^{\an}$ independent of the choice of a skeleton $\Gamma$ of $X^{\an}$, 
 where $\iota: \Gamma \hookrightarrow X^{\an}$ denotes the natural inclusion of $\Gamma$ in $X^{\an}$.

\section{Limit linear series}~\label{sec:lls}
In this section, we discuss the framework of limit linear series from~\cite{A}, 
and  explain how reduction of 
rational functions gives rise to limit linear series in this framework. 
The interested reader is referred to~\cite{A, AB} for more details on relevant materials.

\subsection{Slope structures and linear series on metric graphs}\label{sec:llsmg}

\subsubsection{Rank functions on hypercubes} Denote by $[r]$ the set of integers $0,\dots, r$. 
The hypercube $\Box^d_r$ of dimension $d$ and width $r$ 
is the product $[r]^d$. An element of $\Box^d_r$ denoted by $\underline i$
is thus an integral vector $\underline i = (i_1,\dots, i_d)$ with 
$0\leq i_1,\dots, i_d\leq r$. We denote by $\underline 0$ and $\underline e_m$, $1\leq m \leq d$, 
the origin $\underline 0 = (0,\dots,0)$ and 
the vector whose coordinates are all zero expect the $m$-th coordinate which is equal to one, respectively.

The partial order on $\Box^d_r$ is denoted by $\leq$ and is defined by 
$\underline i \leq \underline j$ if for any $m$, $i_m \leq j_m$. 
There is a lattice structure on $\Box^d_r$
induced by the two operations $\vee$ and $\wedge$ defined by taking the 
maximum and the minimum coordinate-wise, respectively: 
for $\underline i$ and $\underline j$ in $\Box_r^d$, we have 
$$\underline i \vee \underline j = (\max(i_1,j_1),\dots, \max(i_d,j_d)) \qquad \underline i \wedge \underline j = (\min(i_1,j_1),\dots, \min(i_d,j_d)).$$

A function $f: \Box^d_r \rightarrow \mathbb Z$ is called supermodular if for any two elements $\underline i$ and 
$\underline j$, one has
\[f(\underline i) + f(\underline j) \leq f(\underline i\vee \underline j) + f(\underline i\wedge \underline j) .\]
A rank function $\rho : \Box^d_r\rightarrow \mathbb Z$ is a supermodular function satisfying 
the following conditions:
\begin{itemize}
 \item The values of $\rho$ are in the set $[r] \cup \{-1\}$.
 \item $\rho$ is decreasing, i.e., if $\underline i \leq \underline j$, then $\rho(\underline j) \leq \rho(\underline i)$.
 \item $\rho(\underline 0) = r$, and for any 
 $1 \leq m \leq d $, $\rho(\underline e_m) = r-1$. 
\end{itemize}

Applying the supermodularity of $\rho$, one easily sees that 
$\rho(\underline i + \underline e_m) \geq \rho(\underline i) -1 $. 
For a rank function $\rho$ on $\Box^d_r$, we define the set of jumps $J_\rho$ of $\rho$ as follows:
A point $\underline i$ belongs to $J_\rho$, if $\rho(\underline i) \geq 0$ and 
for any $1\leq m \leq d$ 
one has $\rho(\underline i + \underline e_m) = \rho(\underline i)-1$, 
whenever $\underline i + \underline e_m$ belongs  to $\Box^d_r$.
Note that by using the the monotonicity of $\rho$, the data of $J_\rho$ completely determines the rank function.

\medskip

 The function $\rho : \Box_r^d \rightarrow \mathbb Z$ defined by 
\[\rho(i_1,\dots, i_d) = \max\{-1, r- i_1-\dots-i_d\}\] 
is a rank function. 
The set of jumps of $\rho$ consists of all the points $\underline i = (i_1,\dots,i_d)$ of $\Box^d_r$ 
such that $\sum_m i_m \leq r$.
This is called the \emph{standard rank} function.

\subsubsection{Filtered vector spaces and the induced rank functions} \label{sec:filteredrank}
Let $H$ be a vector space of dimension $r+1$ over $\k$. Suppose we are given $d$ 
decreasing filtrations 
$\mathcal F_{\bullet}^{(m)}$, indexed by $1\leq m\leq d$:
\[\mathcal F^{(m)}_0H = H \supsetneq \mathcal F^{(m)}_1H \supsetneq \dots \supsetneq 
\mathcal F_r^{(m)}H =0.\]
The function $\rho:\Box^d_r \rightarrow \mathbb Z$ defined by 
\[\rho(i_1,\dots, i_d):= \dim_\k \bigl(\,
\mathcal F^{(1)}_{i_1}H \cap \dots \cap \mathcal F^{(d)}_{i_d}H\,\bigr) -1,\]
is a rank function.

A typical example is obtained by a set of $d$ distinct
points $x_1,\dots, x_d$ on a smooth proper curve $C$ over $\k$.  For $H \subset \k(C)$ a vector space of rational functions of dimension $r+1$ over $\k$,  
each point $x_m$ provides a decreasing filtration 
$\mathcal F^{(m)}_\bullet H$ as above by looking at the orders of vanishing of 
functions in $H$ at $x_m$, so $\mathcal F^{(m)}_j$ is the vector space of all rational 
functions in $H$ with an order of vanishing at $x_i$ among the first $r+1-j$ possible values in a decreasing order.

\subsubsection{Slope structures of order $r$ on graphs and metric graphs}
Let first $G=(V,E)$ be  a simple graph. We denote by $A$ the set of all 
the oriented edges (arcs) $uv$ for any edge $\{u,v\}$ in $E$ (so both $uv$ and $vu$ belong to $A$). 
For a vertex $v\in V$, we denote by $A_v\subset A$ the set of all the oriented edges 
$vu\in \vec E$ (i.e., 
$\{v,u\} \in E$).  

A slope structure 
$\mathfrak S = \Bigl\{\,S^v; S^e\Bigr\}_{v\in {V}, e\in A}$ of order $r$ on $G$,
or simply an 
$r$-slope 
structure, is  the data of

\begin{itemize}
 \item For any oriented edge $e=uv\in A$ of $G$, a set of integers $s_0^e<s_1^e<\dots<s_r^e$, 
 subject to the requirement that 
 $s^{uv}_i + s^{vu}_{r-i} =0$ for any (unoriented) edge $\{u,v\}$ in $G$. 
 We denote $S^e = \{s_i^e\,|\,i=0,\dots,r\}$.  
 \item For any vertex $v$ of $G$, a rank function $\rho_v$ on $\Box_r^{\val(v)}$. 
 If $J_{\rho_v}$ denotes the set of jumps of $\rho_v$,
 we denote by 
 $S^v \subseteq \prod_{e\in A_v} S^e$, the set of all points $s_{\underline i}$ 
 for $\underline i \in J_{\rho_v}$. 
 Here, for a point $\underline i=(i_e)_{e\in A_v}$ of the hypercube, 
 the point $s_{\underline i}\in \prod_{e\in A_v} S^e $ denotes the 
 point in the product which has coordinate at $e\in A_v$ equal to $s^e_{i_e}$.
\end{itemize}

\medskip

Let now $\Gamma$ be a metric graph. By an $r$-slope structure on $\Gamma$ we mean an 
$r$-slope structure $\mathfrak S$ on a simple graph model $G=(V,E)$ of $\Gamma$.  
We enrich this slope structure on any point of $\Gamma$ as follows. 
For $ x\in \Gamma$, denote by $\T_x(\Gamma)$ the set of all the $\val(x)$ (out-going) tangent vectors to $\Gamma$ at $x$. 
For any point $x$ and $\nu\in \T_x(\Gamma)$, there exists a unique oriented edge $uv$ of $G$ which is parallel to $\nu$. Define 
$S^\nu = S^{uv}$. Also for any point $x\in \Gamma \setminus V$ in the interior of an edge $\{u,v\}$, 
define  $\rho_x$ to be the standard rank function on $\Box^2_r$. In particular, 
$S^x \subseteq S^{uv}\times S^{vu}$ can be identified with
the set of all pairs $(s^{uv}_i, s^{vu}_j)$ with $i+j\leq r$. The resulting slope structure is denoted by $\Bigl\{\,S^x;S^{\nu}\,\Bigr\}_{x\in \Gamma, \nu\in \T_x(\Gamma)}$.

\subsubsection{Rational functions on a metric graph compatible with a slope structure}
Let $\Gamma$ be a metric graph and let 
$\mathfrak S = \Bigl\{S^x; S^\nu\Bigr\}_{x\in \Gamma, \nu \in \T_x(\Gamma)}$ be a 
slope structure of order $r$ on $\Gamma$. 
A continuous piecewise affine function $f:\Gamma \rightarrow \mathbb R$ is said to be 
compatible with $\mathfrak S$ if the  two following natural conditions are verified:
\begin{itemize}
 \item[$(i)$] for any point $x\in \Gamma$ and any tangent direction $\nu\in \T_x(\Gamma)$, the outgoing slope 
 of $f$ along $\nu$ lies in $S^\nu$.
 \end{itemize}
 Denote by $\delta_x(f)$ the vector in $\prod_{\nu \in \T_x(\Gamma)} S^\nu$ which consists of outgoing slopes of 
 $f$ along $\nu \in \T_x(\Gamma)$. Then
 \begin{itemize}
 \item[$(ii)$] for any point $x$ in $\Gamma$, the vector $\delta_x(f)$ belongs to $S^x$.
\end{itemize}
The space of rational functions on $\Gamma$ compatible with $\mathfrak S$ is denoted by
$\mathrm{Rat}(\Gamma;\mathfrak S)$, or simply $\mathrm{Rat}(\mathfrak S)$ if there is no risk of confusion.

For any rational function $f$ on $\Gamma$, the corresponding principal divisor is denoted by 
\[\div(f) = \sum_{x}  \div_x(f) (x), \qquad \div_x(f) := - \sum_{\nu \in \T_x(\Gamma)} \mathrm{slope}_{\nu}(f). \]
Note that there is a sign difference between our definition of the 
divisor of a rational function and  that of~\cite{AB}.

\subsubsection{Linear equivalence of slope structures}
Two slope structures on $\mathfrak S_1 =\Bigl\{\,S_1^x;S^\nu_1\,\Bigr\}_{x\in \Gamma, \nu\in \T_x(\Gamma)}$ 
and 
$\mathfrak S_2=\Bigl\{\,S_1^x;S^\nu_1\,\Bigr\}_{x\in \Gamma, \nu\in \T_x(\Gamma)}$ $\mathfrak S_2$ on a metric graph $\Gamma$
are said to be linearly equivalent, written $\mathfrak S_1 \simeq \mathfrak S_2$, if there exists a rational function $f$ on $\Gamma$ 
 such that for any point $x$ of $\Gamma$ and any $\nu \in \T_x(\Gamma)$, we have $S^\nu_{1} = S^\nu_{2} - 
 \mathrm{slope}_\nu(f)$, and $S^x_{1} = S^x_{2}-\delta_x(f)$. In this case, we write $\mathfrak S_1 = \mathfrak S_2 + \div(f)$.

We extend the definition of linear equivalence between slope structures to 
pairs $(D, \mathfrak S)$ with $D$ a divisor of degree $d$ 
and $\mathfrak S$ an $r$-slope structure on $\Gamma$ by declaring that 
$(D_1,\mathfrak S_1) \simeq (D_2, \mathfrak S_2)$ if there exists a rational function $f$ 
on $\Gamma$ such that 
$D_1 =  D_2 + \div(f)$ and $\mathfrak S_1 = \mathfrak S_2+ \div(f)$. 
\subsubsection{Linear series $\g^r_d$ on $\Gamma$} A $\g^r_d$ on $\Gamma$ is by definition the linear equivalence class 
of the data of a pair $(D, \mathfrak S)$ where $D$ is a divisor of degree 
$d$ on $\Gamma$ and $\mathfrak S$ is an $r$-slope structure on $\Gamma$ subject to the following
property:
\medskip

 $(*)$ For any effective divisor $E$ on $\Gamma$ of degree $r$, there exists a rational function 
 $f \in \mathrm{Rat}(\mathfrak S)$ such that 
 \begin{itemize}
  \item[(1)] For any point $x \in \Gamma$, $\rho_x(\delta_x(f)) \geq E(x)$; and in addition, 
  \item[(2)] $\div(f)+D -E \geq 0$.
 \end{itemize}

If $\mathfrak S$ comes from an $r$-slope structure on a graph model $G=(V, E)$ of $\Gamma$, then  
 for any point $x\in \Gamma \setminus V$ lying on an edge $\{u,v\}$, 
 the rank function $\rho_x$ is standard and the 
 above condition is equivalent to $i+ j \leq r- E(x)$, where $s^{uv}_i$ and $s^{vu}_j$ are the two slopes of $f$ at $x$. In particular, $\div_x(f) \geq 0$ automatically holds for any point $x$ 
 in the interior of an edge of the model $G$. 

\medskip

Let $(D, \mathfrak S)$ define a $\g^r_d$ on $\Gamma$. We denote by $\mathrm{Rat}(D; \mathfrak S)$ the space of 
all $f\in \mathrm{Rat}(\mathfrak S)$ with the property that $\div(f)+ D \geq 0$, and define 
the linear system $|(D,\mathfrak S)|$ associated to $(D, \mathfrak S)$  
as the space of all  effective divisors $E$ on $\Gamma$ of the form 
$\div(f)+D$ for some $f\in \mathrm{Rat}(D; \mathfrak S)$. Note that $|(D, \mathfrak S)|$ 
is independent of the choice of the pair $(D, \mathfrak S)$ in its linear equivalence class.
By an abuse of the notation, 
we refer to both $(D, \mathfrak S)$ and $|(D, \mathfrak S)|$ as a $\g^r_d$ on $\Gamma$.

\subsubsection{Reduced divisors}
Let $(D, \mathfrak S)$ be a $\g^r_d$ on $\Gamma$ and $v$ a point of $\Gamma$. 
The $v$-reduced divisor $D_v \in |(D, \mathfrak S)|$ is defined as follows. 
Define the rational function $f_v$ by 
\[f_v(x):=\min_{f\in \mathrm{Rat}(D;\mathfrak S)} f(x) - f(v)\]
for any point $x$ of $\Gamma$. (In particular, $f_v(v)=0$.) The function $f_v$ 
belongs to $\mathrm{Rat}(D; \mathfrak S)$. The (effective) divisor $D_v \in |(D,\mathfrak S)|$ 
defined by $D_v := D+ \div(f_v)$ is called 
the $v$-reduced divisor linearly equivalent to $D$ with respect to  $\mathfrak S$. 
We have the following useful proposition~\cite{A}.
\begin{prop}\label{prop:coef}
For any point $x\in \Gamma$, the coefficient of $x$ in $D_x$ is equal to 
$D(x) - \sum_{\nu\in \T_x(\Gamma)} s_0^\nu$.  In addition $D(x) - \sum_{\nu\in \T_x(\Gamma)} s_0^\nu \geq r$.  
\end{prop}
Further properties of reduced divisors are given in~\cite{A}.

\subsection{Limit linear series on the skeleton of a Berkovich curve}
Let now $X$ be a smooth proper curve $\K$. Let 
$\mathcal D$ be divisor of degree $ d$ on $X$, and  $(\mathcal O(\mathcal D),H)$ 
be a $\g^r_d$ on $X$. We identify $H$ with a subspace of 
 $ \K(X)$ of dimension $r+1$.  Let $\Gamma$ be a 
skeleton of $X^{\mathrm{an}}$. In this section, we define 
the reduction of $(\mathcal O(\mathcal D), H)$ to $\Gamma$ which will be a $\g^r_d$ on $\Gamma$ 
that we call the limit linear series on $\Gamma$ induced by $(\mathcal O(\mathcal D), H)$, 
or simply the limit $\g^r_d$. First we recall the following basic fact from~\cite{AB}.  
 \begin{lemma}\label{lem:dimensionreduction} Let $X$ be a smooth proper curve over $\K$, and $x\in X^{\mathrm{an}}$  
a point of type 2. The $\k$-vector space $H_x$ defined by the reduction  to $\widetilde{\H(x)}$ 
of an $(r+1)$-dimensional $\K$-subspace 
$H \subset \K(X)$  has dimension $r+1$. 
\end{lemma}

 By the slope formula, the reduction $F=-\log(|f|)$ 
 of any function $f\in H$ to $\Gamma$ is a piecewise affine function on 
 $\Gamma$ with integer slopes. Let first $x$ be a type II point of $ \Gamma$, 
 and $\nu$ a tangent direction in
 $\T_x(\Gamma)$. Denote by $x^\nu$ the point of $\C_x(\k)$ which corresponds to $\nu$.
 By Lemma~\ref{lem:dimensionreduction}, the dimension of $\widetilde H \subset \k(\C_x)$ 
 is $(r+1)$. The orders of vanishing
  of $\widetilde f\in \widetilde H$ at $x^\nu$ define a sequence 
  of integers $s_0^{\nu} < s_1^{\nu} < \dots <s_r^{\nu}$. Denote by $S^\nu = \{s_i^\nu\}$. 
 In addition, the collection of points $x^\nu\in \C_x(\k)$ for $\nu\in \T_x(\Gamma)$ 
 define a rank function $\rho_x$
 associated to the corresponding filtrations on 
 $\widetilde H$ as in Section~\ref{sec:filteredrank}. 
 We define $S^x$ as the set of jumps of $\rho_x$. The proof of the following theorem can be found in~\cite{A}. 
  \begin{thm}[Specialization of linear series] Let $(\mathcal O(\mathcal D),H)$, $H \subseteq H^0\bigl(X, \mathcal O(\mathcal D)\bigr) \subset 
  \K(X)$, be a $g^r_d$ on $X$. Let $\Gamma$ be a skeleton of $X^{\an}$.
  There exists a semistable vertex $V$ for $X$ such that $\Sigma(X,V) =\Gamma$, and such that 
  the slopes of rational functions $f$ in $H$ along edges in $\Gamma$ define a well-defined 
  $\g^r_d$ $(D, \mathfrak S)$
  on $\Gamma$, with $D = \tau_*(\mathcal D)$. 
 \end{thm}
 
We already defined $S^x$ and $S^\nu$ for type II points of 
 $\Gamma \subset X^{\an}$ and $\nu \in \T_x(\Gamma)$. The main point of the above theorem is 
 that the definitions 
 can be extended to all points of 
 $\Gamma$, and that the collection $\mathfrak S = \Bigl\{\,S^x; S^\nu\,\Bigr\}_{x\in \Gamma, \nu\in \T_x(\Gamma)}$ 
 is induced from a simple graph model of $\Gamma$ (or equivalently, from a semistable 
 vertex set of $X^{\an}$). The fact that $\tau_*(\div(f)) = \div(-\log|f|)$, which is a consequence
 of the slope formula, then shows that $(\tau_*(\mathcal D), \mathfrak S)$ 
 is a $\g^r_d$ on $\Gamma$, as in the proof of the specialization theorem for metrized 
 complexes~\cite{AB}; we omit the details which can be found in~\cite{A}. 
 We should say that in particular, for two linearly equivalent divisors 
  $\mathcal D \sim \mathcal D'$ on $X$, and $H$ a subspace 
  of the space of global sections of the corresponding line bundles $\mathcal O(\mathcal D)\simeq \mathcal O(\mathcal D')$ of projective dimension $r$, 
  the above theorem ensures that the two pairs $(D, \mathfrak S)$ and $(D',\mathfrak S')$ are linearly equivalent.   
 
 \medskip
 
  Note that in particular, if $(D, \mathfrak S)$ is a limit linear series on $\Gamma$ 
  induced by a $\g^r_d$ $(\mathcal D, H)$ on $X$, for $H \subseteq H^0(X, \mathcal O(\mathcal D))$,  
  then, by Proposition~\ref{prop:coef},
  for any point $x$ of $\Gamma$, we have 
  \begin{equation}\label{lem:redcoeffineq}
  D(x) -\sum_{\nu \in \T_x(\Gamma)} s_0^\nu \geq r\,.
  \end{equation}

\section{Reduction of Weierstrass divisors}~\label{sec:reduction}
Let $\mathcal D$ be  a divisor of degree $d$ and rank $r$ on $X$. We first 
show that in the case
$\mathrm{char}(\k) =0$, 
the induced limit linear $\g^r_d$ $(D, \mathfrak S)$ on the skeleton $\Gamma$ 
of $X^{\an}$ allows to describe the reduction of Weierstrass divisor of the 
line bundle  $L=\mathcal O(\mathcal D)$ on $X$.  We further 
give a similar, in spirit, description of the reduction in positive characteristic case, 
noting however that, as the form clearly suggests, 
the data of limit linear series is not, in general, enough 
to describe the reduction of Weierstrass divisor when $\mathrm{char}(\k)>0$.

\subsection{Equal characteristic zero}
  The main result of this section is the proof of Theorem~\ref{thm:SpW}. We will give a 
  simple proof here. However,  we note that, strictly speaking,   
the more conceptual approach given in the proof of Theorem~\ref{thm:SpW2} below 
implies this result.
  
 Let $L = \mathcal O(\mathcal D)$ be a line bundle of degree $d$ and rank $r$ on $X$, 
let $\Gamma$ be a skeleton of $X^{\an}$  
 and let  $(D, \mathfrak S)$ be the limit $\g^r_d$ induced on $\Gamma$ by 
 the specialization theorem.
 Denote by $\mathcal W$ the 
 Weierstrass divisor of $L$ on $X$.   Write $D = \sum_{x\in\Gamma} d_x(x)$. We have to show that 
 \[\tau_*(\mathcal W) = \sum_{x\in \Gamma} c_x (x),\]
 where the coefficient $c_x$ of a point $x$ has the following expression:
 \[c_x = (r+1)d_x + \frac{r(r+1)}{2}(2g_x-2+\val(x)) - \sum_{\nu\in \T_x(\Gamma)} \sum_{i=0}^r s_i^\nu.\]

\medskip

We will first prove a local version for standard 
balls  in $\mathbb A^{1,\an}$ as follows. Let $\B$ denote the standard closed ball with the 
ring of analytic functions the Tate algebra $\K\{T\} = 
\Bigl\{\sum_{i=0}^\infty a_i T^i\,\,\, \,|\,\,\,\, |a_i|\rightarrow 0\Bigr\}$, 
 with the supremum norm
\[|\,\sum_{i=0}^{\infty}a_iT^i\,|_{\sup} = \max |a_i|.\]
The reduction of $\K\{T\}$ with respect to $|\,.\,|_{\sup}$ is the polynomial ring $\k[t]$ where $t$ denotes the reduction of $T$. 
Denote by $\zeta$ the corresponding point of $\B$, and by $\B_+$ the open ball of $\B$ whose closure in 
$\B$ corresponds to the point $(t=0)$ of $\Spec \k[t]$. Note that $\B_+(\K) = \{a\in K\,|\, |a|<1\}$.

\medskip

 Let $f_0,\dots, f_r$ be $(r+1)$ $\K$-linearly independent meromorphic functions on $\B$ (each $f_i$
  is of the form $g_i/h_i$ for $g_i,h_i \in \K\{T\}$, with $h_i\neq 0$). 
 Suppose that the reductions at $\zeta$ have orders of vanishing 
 $s_0< \dots < s_r$ at point $0\in \Spec \k[t]$. 
 Let $\mathcal F=\{f_0,\dots, f_r\}$, and denote by $H$ the $\K$-vector space generated by 
 $f_i$s.  To any point $a \in \B_+(\K)$, 
 one associates the increasing sequence $s_0^a<\dots<s^a_r$ of all the orders of 
 vanishing of meromorphic functions in $H$. Define the weight of $a$ with respect to $H$ by 
 $$w(a) = w_{H}(a):=
 s_0^a+\dots+s^a_r - \frac {r(r+1)}2.$$
A point $a\in \B_+(\K)$ is called a Weierstrass point of $H$ if $w(a) \neq 0$. The Wronskian $\mathrm{Wr}_{\mathcal F}$ is the meromorphic function on $\B_+$ defined by 
   $$\mathrm{Wr}_{\mathcal F} := \det \Bigl(f_{i}^{(j)}\Bigr)_{0\leq i,j\leq r}, \qquad  
   \mathrm{where} \qquad f^{(j)}_i = \frac {d^jf_i}{dT^j}.$$
  
The following is straightforward, and holds more generally under the weaker assumption $\mathrm{char}(\K)=0$, i.e.,  
without any restriction on the characteristic of $\k$.
   \begin{lemma}\label{lem:cle} The Weierstrass divisor of $H$  is equal to
the zero divisor of the meromorphic function $\mathrm{Wr}_{\mathcal F}$ in $\B_+$.
\end{lemma}
Let $F = -\log|\mathrm{Wr}_{\mathcal F}|$. Denote by $\nu$ the tangent direction in $\T_\zeta$ 
corresponding to the point $0\in \Spec \k[t]$. Note that 
$F$ is a piecewise affine function with integral slopes on $\B$, which is 
 harmonic on $\B\setminus \zeta$. By the previous lemma, 
 the slope of $F$ along the unique tangent direction $\nu_a\in \T_a$ for $a\in \B(\K)$, more precisely, along the unique ray from $a$ to $\zeta$ in $\B$, 
 is equal to $-w(a)$. In addition, the slope of $F$ along $\nu \in \T_\zeta$ is equal to 
 the order of vanishing at $0$ of the reduction of 
 $\widetilde{\mathrm{Wr}_{\mathcal F}}\in \k(t)$. It thus follows 
 by the slope formula~\cite{BPR} that 
 $$\mathrm{slope}_\nu(F) = \sum_{a\in \B_+(\K)} w(a).$$

\begin{lemma}\label{lem:totalweight}
 The total weight $\sum_{a\in \B_+(\K)}w(a)$ of Weierstrass points with respect to $H$ in 
 $\B_+$ is equal to $s_0+\dots+s_r - \frac{r(r+1)}2$.
\end{lemma}

\begin{proof}
 By the previous lemma, and the above discussion, we need 
 to show that 
 $$\ord_0\Bigl(\,\widetilde{\mathrm{Wr}_{\mathcal F}}\,\Bigr) = s_0+\dots+s_r - \frac{r(r+1)}2.$$
  
This follows from the standard properties of the Wronskian in characteristic zero: write
$\widetilde{f_i}  = t^{s_i} \mathfrak f_i \in \k(t)$, with $(\mathfrak f_i,t)=1$. 
Let $\widetilde{\mathcal F} = \{\widetilde{f_i}\}$.
Then 
$\widetilde{\mathrm{Wr}_{\mathcal F}} = \mathrm{Wr}_{\widetilde{\mathcal F}}$, and we have 
$$\ord_0\Bigl(\,\widetilde{\mathrm{Wr}_{\mathcal F}}\Bigr) 
= \ord_0 \Bigl(\det\Bigl(\,(s_i^{j}t^{s_i-j})_{i,j=0}^r\,\Bigr)\Bigr),$$
from which the result follows. 
\end{proof}

We need one more fact. Let $C$ be a smooth proper curve of genus $g(C)$ 
over an algebraically closed field 
$\k$ of characteristic zero, 
and let $\widetilde{H} \subset \k(C)$ be a vector space of dimension $(r+1)$ over $\k$. 
 For any point $x\in C(\k)$, denote by $s_0^x<\dots<s_r^x$ all the different orders 
 of vanishing of rational functions $f\in \widetilde H$ at $x$. Define the weight of $x$ by 
 $w(x) = \sum_{i=0}^r s_i^x - \frac{r(r+1)}2$. 
\begin{lemma}\label{lem:totweier} Notations as above, we have
$$\sum_{a\in C} w(a) = (g(C)-1) r(r+1).$$
\end{lemma}
\begin{proof} The weight of a point, as defined above, corresponds to the order of vanishing of 
the Wronskian, which
is a meromorphic section of 
$\Omega^{r(r+1)/2}$. The total sum of weights of points of $C$ is thus equal to 
the degree of $\Omega^{r(r+1)/2}$. 
\end{proof}

\begin{proof}[Proof of Theorem~\ref{thm:SpW}] 
Let $H = H^0(X, \mathcal O(\mathcal D))$, that we view as $H \subseteq \K(X)$. 
The complement of $\Gamma$ in $X^{\an}$ is a disjoint union of open balls. 
Each such open ball $B_\nu$ corresponds to a unique 
$\nu \in \T_x \setminus \T_x(\Gamma)$ for a point of type II $x$ in $\Gamma$, and  
is isomorphic to the standard open ball 
$\B_+$. The point $x$ is the end of $B_\nu$, which is the 
unique point of type II in the closure of $B_{\nu}$ in $X^{\an}$ which lies in 
$\Gamma$. Denote by $x^\nu$ the  point of $\C_x(\k)$ corresponding to $\nu\in \T_x\setminus \T_x(\Gamma)$. 

The restriction of $H$ to $B_\nu$ gives an $(r+1)$ dimensional vector 
space of meromorphic functions on $\B_+ \simeq B_\nu$. Let $s_0^\nu<\dots<s_r^\nu$ be the $(r+1)$ different orders of vanishing of the rational functions in 
$\widetilde H$ at $x^\nu$. They correspond to the $(r+1)$ different slopes along $\nu$ 
of the reductions $-\log|f|$ of rational functions $f\in H$.  Thus, 
by Lemma~\ref{lem:totalweight}, for the weights 
$w(a)$ associated to points $a\in B(\k)$ with respect to $H$, we get
\[\sum_{a\in B_\nu(\K)} w(a) = \sum_{i=0}^r s_i^\nu -\frac{r(r+1)}2\,.\]

On the other hand, 
the (Weierstrass) multiplicity of a point $a$ of $B_\nu(\K) \subset X(\K)$ is equal to 
$(r+1)\mathcal D(a) + w(a)$. 
Combining with the above equations, we get
\begin{align*}
 \tau_*(\mathcal W)(x) &= \sum_{\nu\in \T_x \setminus \T_x(\Gamma) \,,\,a\in B_\nu(\K)}
 \Bigl(\,(r+1)\mathcal D(a)+ w(a)\,\Bigr) \\
 &= 
 (r+1)\tau_*(\mathcal D)(x) + \sum_{\nu \in \T_x \setminus \T_x(\Gamma)} \sum_{i=0}^{r}s_i^\nu - \frac{r(r+1)}2\\
&= (r+1)d_x + (g_x-1)r(r+1) - \sum_{\nu\in \T_x(\Gamma)} \Bigl(\,\sum_{i=0}^r s_i^\nu -\frac{r(r+1)}2\,\Bigr) \,\,\, 
\textrm{(By Lemma~\ref{lem:totweier})}\\
&=(r+1)d_x + \frac{r(r+1)}2(2g_x-2 + \val(x)) - \sum_{\nu\in \T_x(\Gamma)}\sum_{i=0}^r s_i^\nu\,,
\end{align*}
which gives the result.
\end{proof}

\subsection{The case $\mathrm{char}(\k)=p>0$}

In this section we state a generalization of
Theorem~\ref{thm:SpW} to arbitrary characteristic. The main  tool is the 
metrization of the sheaf of differential forms, treated in a recent paper of Temkin~\cite{Temkin14} and~\cite{CTT}. We use the definition of Weierstrass points as 
in~\cite{SV, Schmidt, Lak84}, which works well in positive characteristic.

Denote by $0\leq b_0<\dots < b_r$ the gap sequence of $\mathcal O(\mathcal D)$ on $X$. 
By definition, $\mathcal O(\mathcal D)$ is classical if this is 
the sequence $0,1,\dots, r$; this happens for example when $\mathrm{char}(\K)$ 
is of characteristic zero.  Consider now a basis of global sections $f_0,\dots, f_r\in \K(X)$ of $\mathcal O(\mathcal D)$, and let $\mathcal F =\{f_0,\dots,f_r\}$. We have
\begin{thm}\label{thm:SpW2}
 Let $L = \mathcal O(\mathcal D)$ be a line bundle of degree $d$ and rank $r$, 
 and denote by $\mathcal W$ its 
 Weierstrass divisor on $X$. Let $\Gamma$ be a skeleton of $X^{\an}$, 
 and $D = \tau_*(\mathcal D) = \sum_{x\in\Gamma} d_x(x)$. We have 
 \[\tau_*(\mathcal W) = \sum_{x\in \Gamma} c_x (x),\]
 where the coefficient $c_x$ of a point $x$ has the following expression:
 \[c_x = (r+1)d_x + \Bigr(\,\sum_{i=0}^r\,b_i\,\Bigl)\Bigl(\,2g_x-2 +
 \val(x)\,\Bigr) - 
 \sum_{\nu\in \T_x(\Gamma)} \ord_{x^{\nu}} \widetilde{\mathrm{Wr}_{\mathcal F}}.\]
 Furthermore, if $t$ is a tame parameter at $x$ which gives local parameters $t-a_\nu$ 
 at points $x^\nu \in \C_x$, for $\nu\in \T_x(\Gamma)$, and for some $a_\nu\in \k$, then 
 \[c_x = (r+1)d_x + \Bigr(\,\sum_{i=0}^r\,b_i\,\Bigl)\Bigl(\,2g_x-2 \,\Bigr) - 
 \sum_{\nu\in \T_x(\Gamma)} \ord_{x^{\nu}} \widetilde{\mathrm{Wr}_{\mathcal F,t}}\,.\]
\end{thm}

We should say what $\widetilde{\mathrm{Wr}_{\mathcal F}}$ and 
$\widetilde{\mathrm{Wr}_{\mathcal F,t}}$ in the above 
expression stands for, and make them more explicit.  Consider the extension $\H(x)/\K$. 
By the uniformization theorem of Temkin~\cite{Temkin}, 
$\H(x)$ has an unramified parameter, i.e., there exists a parameter 
$t\in \H(x)$ such that $\H(x)/\widehat{\K(t)}$ is a finite unramified extension.  
For $t$  a parameter for $\H(x)/\K$, i.e., an element $t\in \H(x)$ such that 
$\H(x)/\widehat{\K(t)}$ is finite separable, define the radius of $t$ at $x$ as 
$r_t(x) = \inf_{c\in\K} |t-c|_x$. 

Let $|.|_\Omega$  be the K\"ahler seminorm on the module of differentials 
$\Omega_{\H(x)/\K}$ introduced in~\cite{Temkin14},
which is by definition the maximal seminorm such that the derivation
map $d: \H(x) \rightarrow  \Omega_{\H(x)/\K}$ is contracting.  
Denote by $\widehat{\Omega}_{\H(x)/\widehat{\K(t)}}$ the completion of the module of 
K\"ahler differentials with respect to the K\"ahler seminorm $|.|_{\Omega}$.

Let $t$ be a tame parameter for $\H(x)/\K$. Define the Hasse derivative on $\K[t]$  by 
$D^{(j)}(t^m) = \binom{m}{j}t^{m-j}$,  for $m>0$, and extend by linearity. 
This extends uniquely to $\K(t)$ and further to 
$\widehat{\K(t)}$. Since $\H(x)/\widehat{\K(t)}$ is separable, it also extends to  $\H(x)$. Note that 
$j! D^{(j)}$ is the usual derivative $d^{(j)}/dt^j$ defined as follows: for any $f\in \H(x)$, 
since $\Omega_{\H(x)/\widehat{\K(t)}}$ is one dimensional, there exists $h\in \H(x)$ 
such that 
$df = h dt$. Define $\frac{df}{dt} := h$. Then $d^{(j)}/dt^j$ is $j$-times 
composition of $d/dt$.

Let $t$ be a tame parameter at $x$. 
Then there exists an analytic subdomain $Y$ 
of $X^{\an}$ such that $t$ is a tame parameter for any point of $Y$~\cite{CTT}.

\medskip

Consider now the basis  $\mathcal F=\{f_0,\dots, f_r\}$ of the space of global sections of 
$\mathcal O(\mathcal D)$, which we view as elements of $\H(x)$. 
To simplify the presentation, set $b := \sum_{i=0}^r b_i$, and 
define 
\[\mathrm{Wr}_{\mathcal F,t} = \det \Bigl(\,D^{(b_j)}f_i\,\Bigr)_{i,j=0}^r\,,\]
and further define
$\mathrm{Wr}_{\mathcal F} = \mathrm{Wr}_{\mathcal F,t} dt^{b}$. Note that 
$\mathrm{Wr}_{\mathcal F}$ is a well-defined element of 
$\Omega_{\H(x)/\widehat{\K(t)}}^{\otimes b}$, and we have
\[|\mathrm{Wr}_{\mathcal F}|_\Omega^{\otimes b} = r_t^b \,.\, 
|\,\mathrm{Wr}_{\mathcal F,t}\,|_x \, .\]
Define $\ord_{x^{\nu}}\widetilde{\mathrm{Wr}_{\mathcal F}}$
as the slope of the piecewise linear function $-\log|\mathrm{Wr}_{\mathcal F}|$ 
at $x$ along $\nu$. Therefore, we have
\[\ord_{x^{\nu}}\widetilde{\mathrm{Wr}_{\mathcal F}} = 
-b \,\mathrm{slope}_{\nu}\Bigl(\,\log r_t\,\Bigr) + \ord_{x^{\nu}}\mathrm{Wr}_{\mathcal F,t}\,.\]

In particular, if $t$ is a parameter with $-\log(r_t)$ increasing of slope 1 along each 
$\nu$ (if $t-a_\nu$ is a local parameter at $x^\nu \in \C_x$ for some $a_\nu \in \k$), 
then we can rewrite 
\begin{align}
 c_x &= (r+1)d_x + b\,\Bigl(\,2g_x-2 +
 \val(x)\,\Bigr) - 
 \sum_{\nu\in \T_x(\Gamma)} \ord_{x^{\nu}} \widetilde{\mathrm{Wr}_{\mathcal F}} \\
 &= (r+1)d_x+b\,\bigl(\,2g_x-2\,\bigr) - 
 \sum_{\nu \in \T_x(\Gamma)} \ord_{x^\nu} \widetilde{\mathrm{Wr}_{\mathcal F, t}}\,,
\end{align}
which proves the second half of Theorem~\ref{thm:SpW2}. This latter form is the precise analogue of Theorem~\ref{thm:SpW}: indeed, in the case where 
$\mathrm{char}(\k) = 0$, we can easily see 
that 
\[\ord_{x^\nu}\widetilde{\mathrm{Wr}_{\mathcal F, t}} = \sum_{i=0}^r s^\nu_i - r(r+1)/2,\]
which  precisely reduces the statement to that of Theorem~\ref{thm:SpW}. 

\begin{proof}[Proof of Theorem~\ref{thm:SpW2}]
 The proof uses the recent results in~\cite{CTT} concerning the metrization of the sheaf of K\"ahler differentials in 
 $G$-topology. We follow the terminology of~\cite{CTT}, to which we refer for more details on
 basic properties of the K\"ahler seminorm on $\Omega_{X_G}$. 
 Let $\mathcal F$ be a basis for 
 $H^0(X, \mathcal O)$, and consider the meromorphic section of $\Omega_{X_G}^{\otimes b}$ given by the Wronskian 
 $\mathrm{Wr}_{\mathcal F}$. For any point of type II $x$ in $\Gamma$, and $\nu \in \T_x \setminus \T_x(\Gamma)$, 
  consider, as in Lemma~\ref{lem:cle}, 
  the open ball $B_\nu$ in $X^{\an}$. By the proof of that 
  lemma~\ref{lem:cle}, and the choice of parameter $T$, 
  the slope of $r_T$ along $\nu$ is one, and thus 
the slope of $-\log|\mathrm{Wr}_{\mathcal F}|$ 
 along $\nu\in \T_x \setminus \T_x(\Gamma)$ minus $b=\sum_{i=0}^r b_i$ is precisely
 the total weight of Weierstrass points in $B_\nu(\K)$. In other words, for any $\nu \in \T_x\setminus \T_x(\Gamma)$, we have
 
 \begin{equation}\label{eq:slope}
  \sum_{a\in B_\nu(\K)} w(a) = \mathrm{slope}_\nu \Bigl(\,-\log|\mathrm{Wr}_{\mathcal F}|\,\Bigr) 
 -b\,.
 \end{equation}

 Let $\Omega^{\diamond}_{X_G}$ be the unit ball of $|.|_\Omega$, and 
 consider the restriction $\Omega^{\diamond}_{X_G, \C_x}$ of $\Omega^\diamond_{X_G}$ to $\C_x$. 
 The reduction of the twist $\Omega^\diamond_{X_G,\C_x}(-\C_x)$ 
 is the sheaf of differentials $\Omega_{\C_x/\k}$. In addition, the number 
 $\mathrm{slope}_\nu \Bigl(\,-\log|\mathrm{Wr}_{\mathcal F}|\,\Bigr) - b$ 
 is precisely the order of vanishing at $x^{\nu}$ of the reduction 
 $\widetilde{\mathrm{Wr}_{\mathcal F}}$ seen as a meromorphic section of 
 $\widetilde{\Omega_{X_G, \C_X}(-\C_x)}^{\otimes b}$. Since this latter sheaf is isomorphic to 
 $\Omega_{\C_x/\k}^{\otimes b}$, which has total degree $b(2g_x-2)$, it follows that 
 
 \begin{align*}
  b(2g_x-2)\,\,&=\,\,\sum_{\nu \in \T_x} \Bigl(\,\,\mathrm{slope}_\nu \Bigl(\,-\log|\mathrm{Wr}_{\mathcal F}|\,\Bigr) 
 -b\,\,\Bigr) =\,\,\sum_{\nu\in \T_x} \bigl(\,\ord_{x^{\nu}} \widetilde{\mathrm{Wr}_\mathcal F}-b\,\bigr)\,.
 \end{align*}
 On the other hand, the coefficient $c_x$ of $x$ in the reduction $\tau_*(\mathcal W)$ is 
 \[c_x= (r+1)d_x+\sum_{\nu\in \T_x\setminus \T_x(\Gamma)} \sum_{a\in B_\nu(\K)} w(a),\]
which, combined with Equation~\eqref{eq:slope}, gives
\begin{align*}
 c_x &= (r+1)d_x+b(2g_x-2) - \sum_{\nu \in \T_x(\Gamma)}\bigl(\,\ord_{x^{\nu}} \widetilde{\mathrm{Wr}_\mathcal F}-b\,\bigr)\\
 &= (r+1)d_x + b\bigl(\,2g_x-2+ \val(x)\,\bigr) - 
 \sum_{\nu \in \T_x(\Gamma)}\,\ord_{x^{\nu}} \widetilde{\mathrm{Wr}_\mathcal F}\,.
\end{align*}
 This finishes the proof.   
\end{proof}

\section{Local equidistribution theorem}\label{sec:equilocal}

In this section we state a local equidistribution theorem which will be essential in the proof 
of our main theorem. 
 
 Let $X$ be a smooth proper curve over $\K$, $L = \mathcal O(\mathcal D)$ 
 a line bundle of positive degree $d$ on $X$, and 
 $x$ a point of type II in $X^{\an}$. Let $\nu \in \T_x$, and denote 
 by $x^\nu$ the corresponding $\k$-point of the curve 
 $\C_x$.

 For each integer $n\in \mathbb N$, let
 $H_n=H^0(X, \mathcal O(n\mathcal D)) \subseteq \K(X)$, $r_n = h^0(X, n\mathcal D)-1$, 
 and denote by $S_{n}^{\nu}$ the set of all the slopes of the 
 reductions of rational functions in 
 $H_n$ in $X^{\an}$ along $\nu$. Denote by $\widetilde{H_n}$ the reduction of $H_n$ in $\widetilde{\mathcal H}(x)$
  which is a vector space of dimension 
 $r_n+1$ over $\k$.

Since we assume that $\K$ is algebraically closed, by~\cite[Lemma 2.1.15]{MS}, the valuation of $\K$ has a (non-necessarily canonical) section. Fixing such a section, 
we have a canonical reduction which satisfies
 $\widetilde{f.g} = \widetilde f.\widetilde g$, and thus $\widetilde{H_n}.\widetilde{H_m} 
 \subseteq \widetilde{H_{n+m}}$, 
 for all $n,m\in\mathbb N$. So we can define the graded algebra, 
 reduction of the sectional ring of $\mathcal O(\mathcal D)$ at $x$, 
 \[\mathfrak A_{\nu} := \bigoplus_{n=0}^\infty \, \widetilde{H_n}\,.\]

 Denote by $\Lambda = \Lambda^\nu$ the convex hull in $\mathbb R$ 
 of all the rational numbers of the form $\frac sn$ for $n\in \mathbb N$ and $s\in S_n^\nu$, 
 \[\Lambda:= \textrm{conv-hull}\,\Bigl(\bigcup_{n\in \mathbb N, f\in H_n} 
 \bigl\{\frac {\mathrm{slope}_{\nu}(-\log|f|)}{n}\bigr\}\Bigr) = 
 \textrm{conv-hull}\,\Bigl(\bigcup_{n\in \mathbb N, \widetilde f\in \widetilde{H_n}}  
 \bigl\{\frac {\ord_{x^\nu}(\widetilde f)}{n}\bigr\}\Bigr).\]
Note that $\Lambda$ 
is precisely the Okounkov body associated to the graded ring 
of rational functions $\mathfrak A_{\nu}$ 
and the divisor  $\{x^\nu\} \subset \C_x$~\cite{KK, LM}.

Since the dimension of $\widetilde H_n$ (=$|S^{\nu}_n|$) is equal to 
$r_n+1 = dn-g+1$ for 
 large enough $n$, we have $\mathrm{vol}(\Lambda)  = d$, which shows that $\Lambda$ is 
 an interval of the form 
 $[s^\nu_{\min}, s^\nu_{\max}]$ for real numbers $s^\nu_{\min}$ and $s^\nu_{\max}$
with $s^\nu_{\max}=s^\nu_{\min}+d$ (we drop $\nu$ if there is no risk of confusion).

 \medskip
 
 For any $n\in \mathbb N$, denote now by $\eta_n$ the discrete probability measure on 
 $\Lambda \subset \mathbb R$
 \[\eta_n:=\frac 1{r_n+1} \sum_{s\in S_n^\nu}\delta_{n^{-1}s}.\]
The following is a special case of a more general equidistribution theorem for convex bodies 
associated to semigroups~\cite{Ok, KK, Bouk}, see e.g.,~\cite[Th\'eor\`eme 0.2]{Bouk}\,.
 \begin{thm}[Local equidistribution]\label{thm:equilocal}
The measures $\eta_n$ converge weakly to the Lebesgue measure on $\Lambda$. 
\end{thm}

 A direct consequence of the above theorem is the following useful corollary. 

\noindent Denote by $s_{n,0}^\nu < \dots<s_{n,r_n}^{\nu}$ all the integers in $S^\nu_n$.
First note that for two integers $n,m$ we have 
\[s^\nu_{n+m,0} \leq s^\nu_{n,0}+ s^\nu_{m,0}\,, \,\,\, \textrm{and} \,\,\, 
 s^\nu_{n,r_n}+ 
s^\nu_{m,r_m} \leq s^\nu_{n+m,r_{n+m}} \,, \]
which shows by Fekete lemma that
 \[\lim_{n\rightarrow \infty} \frac 1n s^\nu_{n,0} =s_{\min}^\nu\,,\,\,\textrm{and}\,\,\,
 \lim_{n\rightarrow \infty} \frac 1n s^\nu_{n,r_n} =s_{\max}^\nu\,. \]
 In particular, 
 \[\lim_{n \rightarrow \infty }\frac{s^\nu_{n,0} + s^{\nu}_{n,r_n}}{2nd} = 
 \frac{s_{\min}+s_{\max}}{2d} = \frac1d 
 \int_{\Lambda} \ta\,d\ta\,,\]
 where $d\ta$ is the Lebesgue measure on $\Lambda$. 
\begin{cor}\label{cor:sminmax}
 Notations as above, we have 
 \[\lim_{n\rightarrow \infty} \Bigl[\, 
  \frac{s^\nu_{n,0} + s^{\nu}_{n,r_n}}{2nd}  - \frac 1{nd(r_n+1)} \sum_{s\in S^\nu_{n}}s \,\Bigr] =0.\]
\end{cor}

\begin{remark}\label{rem:positive}\rm
 Consider the case  $\mathrm{char}(\k)>0$. The equidistribution theorem for Weierstrass points 
 would then follow, by the proof in Section~\ref{sec:proof}, 
 if we could guarantee the following analogue of Corollary~\ref{cor:sminmax}. Notations as in 
 Theorem~\ref{thm:SpW2}, let $t$ be a tame parameter at 
 $x$ such that $t-a_\nu$ reduces to a local parameter at any point 
 $x^\nu \in \C_x$, for $\nu\in \T_x(\Gamma)$ and some
 $a_\nu\in \k$, then 
 \[\frac 1{(r_n+1)^2}\,\ord_{x^{\nu}}\mathrm{Wr}_{\mathcal F_n,t}\,\, 
 \longrightarrow \,\,\frac {s^\nu_{\min} + 
 s^\nu_{\max}}{2d} -\frac 12\,,\]
 where $\mathcal F_n \subset \K(X)$ is an arbitrary basis of $H^0\bigl(\,X, \mathcal O(n\mathcal D)\,\bigr)$.
 
\end{remark}

 \section{Proof of Theorem~\ref{thm:main}}\label{sec:proof}
In this section, we use the materials of the previous sections to prove Theorem~\ref{thm:main}. 
So let $\mathcal D$ be a divisor of positive degree on a smooth proper connected curve $X$ of positive 
genus $g$ over 
$\K$, and denote by $\mathcal W_n$ the Weierstrass divisor of $\mathcal O(n\mathcal D)$ on $X$. 
Let $D =\tau_*(\mathcal D)$, and $W_n = \tau_*(\mathcal W_n)$ be 
the reductions  on $\Gamma$. Denote by $r_n$ the rank of $n\mathcal D$, 
and note that $W_n $ and $ \mathcal W_n$ are of
degree $g(r_n+1)^2$. 

We suppose that $\mathrm{char}(\k)=0$ (however see Remark~\ref{rem:positive}). Let $\Gamma$ be a skeleton of $X^{\an}$ with a simple graph model $G=(V,E)$ which contains all the points 
in the support of $D$ as vertices, and consider 
the limit linear series $(nD, \mathfrak S_n)$ on $\Gamma$ induced from the linear series
$(n\mathcal D, H_n)$ on $X$, with $H_n= H^0(X, \mathcal O(nD))$. 

\medskip

Let $D = \sum_{x\in \Gamma} d_x (x)$. By Theorem~\ref{thm:SpW}, we have 
$W_n = \sum_{x\in \Gamma} c_x (x),$ where 
 \[c_x = (r_n+1)nd_x + \frac{r_n(r_n+1)}{2}(2g_x-2+\val(x)) - \sum_{\nu\in \T_x(\Gamma)} 
 \sum_{i=0}^{r_n} s^{\nu}_{n,i}.\]
 Here, as in the previous section,
 $s^\nu_{n,0}<\dots<s^\nu_{n,r_n}$ are the set of all the integers in $S^\nu_n$.
\medskip

We have to show that the sequence of discrete measures 
$\frac{1}{g(r_n+1)^2} \delta_{W_n}$ converge weakly to the canonical admissible measure 
$\mu_{\ad}$. 

Recall that $K_g = \sum_{x\in \Gamma}g_x(x)$, and $K_\Gamma = \sum_{x\in \Gamma} (\val(x)-2)(x)$. 
For any $n$ and $x\in \Gamma$, let
\[p_{n,x} := \sum_{\nu\in \T_x(\Gamma)}\sum_{i=0}^{r_n}s^\nu_{n,i}\,\,\,\,,\,\,\textrm{and}\,\,\,\,\, 
P := \sum_{x\in \Gamma} p_{n,x}(x)\,.\]  

So we can rewrite
\[W_n = (r_n+1)nD + r_n(r_n+1) K_g + \frac{r_n(r_n+1)}{2} 
K_\Gamma - P_n\,,\,\,\textrm{and}\]
\[\frac{1}{g(r_n+1)^2}\delta_{W_n} = \frac n{g(r_n+1)} \delta_{D} +
\frac{r_n}{g(r_n+1)} \delta_{K_g}+ 
\frac{r_n}{2g(r_n+1)}\delta_{K_\Gamma} - \frac{1}{g(r_n+1)^2} \delta_{P_n}\,.\]

\medskip

Certain terms in the right hand side of the above equation have simple asymptotic: 

\begin{prop} We have as $n$ goes to infinity
 \[\frac n{g(r_n+1)}\delta_{D} \longrightarrow \frac1{gd}\delta_D\,,\quad 
  \frac{r_n}{g(r_n+1)} \delta_{K_g} \longrightarrow \frac 1g \delta_{K_g}\,,\,\textrm{and}\quad 
  \frac{r_n}{2g(r_n+1)}\delta_{K_\Gamma} \longrightarrow \frac 1{2g}  \delta_{K_\Gamma}\,.
\]
\end{prop}

Comparing with the explicit form of $\mu_{\ad}$, and observing that 
$r_n = nd-g \sim nd$ for large enough $n$, we see that  

\begin{prop} Notations as above, the following assertions are equivalent:
\begin{equation*}
(*)\qquad \frac 1{g(r_n+1)^2} \delta_{W_n} \longrightarrow \mu_{\ad}, \quad \textrm{weakly}\,.
\end{equation*}

\begin{equation*}
(\star)\qquad \frac 1d\delta_D+\frac 12\delta_{K_\Gamma} - \frac 1{r_n nd} \delta_{P_n} \longrightarrow 
\sum_{e\in E(G)} \frac{1}{\ell_e+\rho_e} d\ta, \quad \textrm{weakly}\,. 
\end{equation*}
\end{prop}

The advantage of ($\star$) is that (almost) everything is formulated in terms of the metric graph $\Gamma$, which makes things 
more flexible to work.

We choose for each $n$ a simple graph model
$G_n$ of $\Gamma$ with vertex set $V_n$, which we suppose to be a refinement of $G$, i.e., $V\subseteq V_n$, 
and which we suppose also to be a model for 
the limit linear series $(nD, \mathfrak S_n)$. 
 
 We will later explain how to easily get Theorem~\ref{thm:main} from 
the following particular case:
\begin{thm}\label{thm:reduction} Let $e = \{x,y\}$ be an edge of a simple graph model 
$G=(V,E)$ of $\Gamma$, for $x,y\in V$. 
Consider the characteristic function 
$1_e$ of $e$. Then
 \[\frac1d\int_{\Gamma}1_e\delta_D + \frac 12 \int_{\Gamma} 1_e \delta_{\mathcal K_\Gamma} - \lim_{n\rightarrow \infty} 
 \frac 1{r_n nd} \int_{\Gamma} 1_e \delta_{P_n} = \int_e \frac {1}{\ell_e+\rho_e}\,d\ta \,=\, 
 \frac{\ell_e}{\ell_e+\rho_e}\,.\]
\end{thm}

\medskip
Note that 
\begin{equation}\label{eq:pr0}
 \frac 12 \int_\Gamma 1_e \delta_{\mathcal K_\Gamma} = 
\frac 12\,\bigl(\,\val(x)+\val(y)\,\bigr) -2\,\,\,\,\,\,\textrm{and}\,\,\,\,\,\,
\frac1d\int_{\Gamma}1_e\delta_D=\frac1d\Bigl(d_x+d_y\Bigr)\,.
\end{equation}

\medskip

We now make the last integral in the left hand side of the above equation  more explicit. Fix an integer 
$n$, and let $u_0=x, u_1, \dots, u_l, u_{l+1}=y$ be the ordered set of all 
vertices of $G_n$ between $x$ and $y$ in $V_n$ on the segment $[x,y]$ of $\Gamma$; thus,  
$u_0u_1u_2\dots u_{l+1}$  is a path in $G_n$. 

Since $G_n$ is a model for the limit linear series $(nD, \mathfrak S_n)$, 
for any point $z\in [x,y]$ different from $u_i$ 
we have $p_{n,z} = 0$ (indeed, $S_{n,z}^\nu = -S_{n,z}^{\bar \nu}$ where 
$\T_z(\Gamma)=\{\nu,\bar\nu\}$). 
Thus the restriction of the divisor 
$P_n$ to the interval $[x,y]$ has support in $\{u_0,u_1,\dots, u_l,u_{l+1}\}$.

This shows that 
\begin{equation}\label{eq:pr1}
 \int_\Gamma 1_e \,\delta_{P_n} = \sum_{z\in V_n \cap [x,y]}\,p_{n,z}\,.
\end{equation}

Furthermore, 
for each edge $\{u,v\}$ of $E(G_n)$ with extremities $u,v\in \{u_0,u_1,\dots,u_l,u_{l+1}\}$, we have 
\[\sum_{s\in S^{uv}_n} s = - \sum_{s\in S_n^{vu}} s\,.\]
Thus, denoting by $\nu_y\in \T_x(\Gamma)$ (resp. $\nu_x\in \T_y(\Gamma)$) 
the tangent direction parallel to $xy$ (resp. $yx$),
we get 
\begin{equation}\label{eq:pr2}
 \sum_{z\in V_n\cap[x,y]} p_{n,z} = q^{\nu_y}_{n,x}  +  q^{\nu_x}_{n,y}\,,
\end{equation}
where for any point $z$ of $\Gamma$, and any tangent direction $\nu$ 
to $\Gamma$ at $z$, we denote by $q_{n,z}^\nu$ the following sum
\[q_{n,z}^\nu := p_{n,z} - \sum_{i=0}^{r_n}s_{n,i}^{\nu}\,.\]

\medskip

For each $\nu \in \T_x(\Gamma) \setminus \{\nu_y\}$, denote, as before, by $x^\nu$ 
the corresponding point
on the curve $\C_x$ associated to $\widetilde{\mathcal H_u}$. 
By Corollary~\ref{cor:sminmax} 
to the 
local equidistribution theorem~\ref{thm:equilocal}, we get 
\begin{equation}\label{eq:pr3}
\lim_{n\rightarrow \infty}\frac 1{r_nnd}\sum_{s\in S_{n}^\nu}s \,=\, 
\lim_{n\rightarrow \infty} \frac{s_{n,0}^\nu+s_{n,r_n}^{\nu}}{2nd}\,=\, \frac 12 + 
\lim_{n\rightarrow \infty} \frac {s^\nu_{n,0}}{nd}\,.
\end{equation}
In the last equality in the above equation, we used the fact that 
\[\lim_{n\rightarrow \infty} \frac{s^\nu_{n,r_n} - s^\nu_{n,0}}{n} = \mathrm{vol}(\Lambda^\nu) = d,\]
where as in  Section~\ref{sec:equilocal},
$\Lambda^\nu$ is the Okounkov body associated to the point $x^\nu$ of $\C_x$. 
\medskip

Combining Equations~\eqref{eq:pr0}, \eqref{eq:pr1}, \eqref{eq:pr2}, \eqref{eq:pr3}, 
and noting that 
$\T_x(\Gamma)$ and $\T_y(\Gamma)$ are of order $\val(x)$ and $\val(y)$, respectively, we see that 
in order to prove Theorem~\ref{thm:reduction}, it will be enough to show that 
\begin{equation}\label{eq:red2}
\frac {d_x+d_y}d-1 - \lim_{n\rightarrow \infty}\frac{1}{nd}\,
\Bigl[\sum_{\nu\in \T_u(\Gamma)\setminus \{\nu_y\}} s^\nu_{n,0} +
\sum_{\nu\in \T_v(\Gamma)\setminus\{\nu_x\}}s^\nu_{n,0} \Bigr] = \frac{\ell_e}{\ell_e+\rho_e}.
\end{equation}

Furthermore, for any point $u$ of $\Gamma$, by Proposition~\ref{prop:coef} and Inequality~\ref{lem:redcoeffineq}, 
the coefficient of the $u$-reduced divisor 
linearly equivalent to $nD$ with respect to $(nD, \mathfrak S_n)$ is $nd_u - \sum_{\nu\in \T_u(\Gamma)} s^\nu_{n,0}$ which is at least
$r_n$ and at most $dn$. Thus,
\begin{equation}\label{eq:pr4}
\lim_{n \rightarrow \infty}  \frac 1{nd}\,\Bigl(\,nd_u -\sum_{\nu\in \T_u(\Gamma)} s^\nu_{n,0}\,\Bigr) 
 = 1,
\end{equation}
which finally reduces Equation~\eqref{eq:red2} to showing that 
\begin{equation}\label{eq:pr5}
 1 + \lim_{n\rightarrow \infty}\frac{1}{nd}\,\Bigl(\,s^{\nu_y}_{n,0}\,+\, 
 s^{\nu_x}_{n,0}\,\Bigr) = \frac{\ell_e}{\ell_e+\rho_e}\,.
\end{equation}

\begin{proof}[Proof of Theorem~\ref{thm:reduction}]
We will prove \eqref{eq:pr5}, which by the above discussion proves the theorem. 

Denote by $D_{n,x}$ be the $x$-reduced divisor linearly equivalent to $nD$ with respect to 
$(nD, \mathfrak S_n)$, and denote by $f_{n,x} 
\in \mathrm{Rat}(D, \mathfrak S_n)$ the rational function with 
$D_{n,x} = nD + \div(f_{n,x})$. Similarly, let $f_{n,y} \in \mathrm{Rat}(D, \mathfrak S_n)$ be 
the rational function which gives 
the $y$-reduced divisor with respect to 
$(nD, \mathfrak S_n)$. Let $f_n := f_{n,y}-f_{n,x}$. We obviously have 
\[D_{n,x}+ \div(f_n) = D_{n,y}.\]

We consider the slopes of the above defined functions on the interval $[x,y]$. First note that by 
Proposition~\ref{prop:coef}, 
\begin{itemize}
 \item The slope of $f_{n,x}$ at $x$ along $\nu_y\in \T_x(\Gamma)$ is $s^{\nu_y}_{n,0}$. 
\end{itemize}
For large enough $n$, by Inequality~\ref{lem:redcoeffineq}, 
the coefficient of $x$ in $D_{n,x}$ is at least $r_n =nd-g$, which given that 
$D_{n,x}$ is effective, shows that the sum of the coefficients of $D_{n,x}$
at the points in the open interval 
$(x,y)$ is at most $g$. Since $nD$ does not have support in $(x,y)$, this implies that  
\begin{itemize}
 \item The slope of $f_{n,x}$ at $y$ along $\nu_x \in \T_y(\Gamma)$ is between 
 $-s^{\nu_y}_{n,0}$ and $g - s^{\nu_y}_{n,0}$.
\end{itemize}
(Recall that $s^{\nu_y}_{n,0}$ is the slope at $x$ along $\nu_y$.)

\medskip
The above reasoning applies to $f_{n,y}$ as well, which gives
\begin{itemize}
 \item The slope of $f_{n,y}$ at $y$ along $\nu_x\in \T_y(\Gamma)$ is $s^{\nu_x}_{n,0}$. 
 Furthermore, the slope of $f_{n,y}$  at $x$ along $\nu_y \in \T_x(\Gamma)$ is
 between $-s^{\nu_x}_{n,0}$ and $g - s^{\nu_x}_{n,0}$\,.
\end{itemize}

These observations together imply that 
\begin{claim}\label{claim} The slope of $f_n = f_{n,y} -f_{n,x}$ at $y$ 
along $\nu_x$ satisfies
\[s^{\nu_x}_{n,0}+s^{\nu_{y}}_{n,0} -g \,\leq\, 
 d_{\nu_x}f_n(y) \,\leq\, s^{\nu_x}_{n,0}+s^{\nu_{y}}_{n,0}\,.\]
Similarly, for the slope of $f_n$ at $x$ along $\nu_y$, we have 
\[-s^{\nu_y}_{n,0}-s^{\nu_x}_{n,0} \,\leq\, d_{\nu_y}f_n(x) \,\leq\, -s^{\nu_y}_{n,0}-
s^{\nu_x}_{n,0}
+g\,.\]
\end{claim}

\medskip

Consider now the restriction $F_n$ of $f_n$ to the metric graph $\Gamma\setminus e$ obtained by 
removing the (interior of the) edge $e$ from $\Gamma$. 
Write
\[\Delta(F_n) = a_y \delta_y - a_x \delta_x + \sum_{z\in \Gamma\setminus [x,y]} a_z \delta_z\,. \]

Since $D_{n,x}+ \div(f_n) = D_{n,y}$, we get 
\[a_y = D_{n,y}(y) - D_{n,x}(y)+d_{\nu_x}(f_n)(y)\,.\]
We use one last time the inequalities
\[nd-g\leq D_{n,y}(y) \leq nd \qquad,\qquad nd-g\leq D_{n,x}(x) \leq nd\,,\]
and the fact that $D_{n,y}$ is effective of degree $dn$, to get $0\leq D_{n,y}(x) \leq g$. 
So all together, we have
\begin{equation}\label{ineq1}
 nd+s^{\nu_x}_{n,0} + s^{\nu_y}_{n,0}-3g \leq  a_{y} \leq  nd +s^{\nu_x}_{n,0} + s^{\nu_y}_{n,0}\,.
\end{equation}
Similarly, we have  $a_x=D_{n,x}(x)-D_{n,y}(x) -d_{\nu_y}(f_n)(x),$ and so 
\begin{equation}\label{ineq2}
 nd+s^{\nu_x}_{n,0} + s^{\nu_y}_{n,0}-3g \,\leq\,  a_{x} \,\leq\,  nd +s^{\nu_x}_{n,0} + 
 s^{\nu_y}_{n,0}\,.
\end{equation}
In addition, we have  
\begin{equation}\label{ineq3}
\textrm{for all $z\in \Gamma\setminus [x,y]$},\,\,\, -g \leq a_z \leq g\,,\,\, \,\,\textrm{and}
\end{equation} 
\begin{equation}\label{ineq4}
 -2g\,\leq\, \sum_{z\in \Gamma\setminus [x,y]} a_z\,\leq\, 2g\,.
\end{equation}

Define $t_n := s^{\nu_x}_{n,0} + s^{\nu_y}_{n,0}$.
We consider two cases depending on whether $\liminf_n\{|nd+t_n|\}$ is finite or not.

\subsubsection*{First case.} \emph{Suppose that $|nd+ t_n| \rightarrow \infty$ as $n$ goes to infinity.}
\medskip

We explain how to finish the proof of Theorem~\ref{thm:reduction}:

\noindent Combining Inequalities~\eqref{ineq1}, \eqref{ineq2}, \eqref{ineq3}, and \eqref{ineq4}, 
we see that 
$\frac 1{nd+ t_n}F_n$ satisfies 
\[\Delta(\frac {F_n}{nd+t_n}) = \delta_y -\delta_x + \alpha_n\,,\]
where $\alpha_n$ is a discrete measure supported on $\Gamma\setminus e$ 
with $\alpha_n \rightarrow 0$ as $n$ goes to infinity. 
Normalizing $F_n$ to have $F_n(x)=0$ if necessary, 
we conclude the pointwise convergence of  
\[\frac {F_n}{nd+t_n} \longrightarrow F\,,\]
where $F$ satisfies 
\[\Delta(F) =\delta_y -\delta_x\,.\]
It thus follows that 
\begin{equation}\label{eq:final}
 \frac 1{nd+t_n}\Bigl(F_n(x) -F_n(y)\Bigr) \longrightarrow F(x)-F(y) = \rho_e\,.
\end{equation}

To conclude, write
\[F_n(x) - F_n(y) = f_n(x)-f_n(y) = \int_y^x \frac d{d\ta}f_n(\ta)\,d\ta\,,\]
and note that, since $D_{n,x} + \div(f_n) =D_{n,y}$ is effective, and the sum of 
all the coefficients of the points on the interval $(x,y)$ is at most $g$, we have
\begin{equation}\label{eq:miracle}
 -t_n-g=s^{\nu_x}_{n,0} + s^{\nu_y}_{n,0} -g 
\leq \frac d{d\ta}f_n(x+\ta) \leq  s^{\nu_x}_{n,0} + s^{\nu_y}_{n,0}+g = - t_n +g,
\end{equation}
which gives 
$$ (-t_n-g)\ell_e\leq F_n(x)-F_n(y) \leq (-t_n+g)\ell_e.$$
Plugging this into Equation~\eqref{eq:final}, we get 
\[\frac{t_n}{nd+t_n} \longrightarrow \frac {-\rho_e}{\ell_e}\,,\]
which finally gives:
\[\lim_{n\rightarrow \infty} \frac 1{nd} \Bigl(s^{\nu_x}_{n,0}+s^{\nu_y}_{n,0}\Bigr) = \frac{-\rho_e}{\ell_e+\rho_e}. \]
Adding one to the both sides of the above equation gives~\eqref{eq:pr5}, and 
finishes the proof.

\medskip
\subsubsection*{Second case.} \emph{Suppose that $|n_id+ t_{n_i}| < C$ for a constant $C$ and a 
sequence $n_i \rightarrow \infty$.}
\medskip

In this case, the coefficients of $\Delta(F_{n_i})$ will be all bounded by a constant, for all $i$,
and the total (positive) degree 
of $\Delta(F_{n_i})$ is also bounded by $2C+g$. We will prove that $e$ is a bridge in $\Gamma$. 
For the sake of a contradiction, suppose this is not the case, i.e., $\Gamma \setminus e$ 
is connected. We can further assume that $F_n(x)=0$ for all $n$. 
By the compactness of $\Gamma$, it then follows that $F_{n_i}$ are all bounded functions on 
$\Gamma$. In other words, 
there exists $C'$ such that $F_{n_i}(z) < C'$ for all $z\in \Gamma \setminus e$. 
By Equation~\eqref{eq:miracle}, we find that the slope of $f_{n_i}$ at any point 
on the segment $(y,x)$ is between $-t_{n_i} -g$ and $-t_{n_i} +g$, which is at least 
$\frac {n_i d}2$, for large $i$. It follows that $f_{n_i}(y) -f_{n_i}(x) \rightarrow \infty$, 
as $i$ goes to infinity,
which is a contradiction. This shows that $e$ is a bridge edge in $\Gamma$. 

We conclude that in this case, the right hand side of~\eqref{eq:pr5} is zero (since $\rho_e =\infty$), 
and it is easy to check that the left hand side is also zero. Indeed, in this case, 
we will  have 
$\Lambda^{\nu_y} = - \Lambda^{\nu_x}$. Now, using our previous notation 
$\Lambda^\nu = [s^\nu_{\min},s^{\nu}_{\max}]$, with $s^{\nu}_{\max} =d + s^{\nu}_{\min}$, we get
\[1+\lim_{n \rightarrow \infty} \frac 1{nd}(s^{\nu_x}_{n,0}+s^{\nu_y}_{n,0}) = \frac 1d 
\Bigl(\,d+s^{\nu_x}_{\min}+
s^{\nu_y}_{\min}\,\Bigr) = s^{\nu_x}_{\max}+s^{\nu_y}_{\min} = 0.   \]
This finally finishes the proof of our theorem.

\end{proof}

\begin{proof}[Proof of Theorem~\ref{thm:main}] This follows as usual from Theorem~\ref{thm:reduction}.
 Let $f$ be a continuous function on $\Gamma$. Take a simple graph model $G$ 
 of $\Gamma$ where each edge has length at most $\epsilon$, and such that the support of $D$ is contained in $V(G)$.
 Choose a point $x_e$ in the interior of each edge $e$, and consider the function 
 $f_\epsilon=\sum_e f(x_e)1_e.$ By Theorem~\ref{thm:reduction}, we have 
 \[\lim_{n\rightarrow \infty} \int_\Gamma f_{\epsilon} d\mu_n = \int_\Gamma f_\epsilon d\mu_{\ad}.\]
 Since $||f_\epsilon -f||_{\infty} \leq \epsilon\,,$ we conclude that 
 for all sufficiently large $n$, we have 
 \[-3\epsilon\leq \int_\Gamma fd\mu_n - \int_\Gamma fd\mu_{\ad} \leq 3\epsilon,\]
 which proves the theorem.
\end{proof}

\begin{proof}[Proof of Corollary~\ref{cor:onecomponent}]
We can pass from $K$ to the completion $\K$ of an algebraic closure of $K$, 
and reformulate the statement in $X_{\K}^{\an}$ by looking at the type II point 
$x_0$ 
which corresponds to $\C_0$ on a fixed 
skeleton $\Gamma$. The assertion, that, the proportion of Weierstrass points which are mapped to 
$x_0$, under the retraction map $\tau$, goes to $g_0/g$ as $n\rightarrow \infty$,  
can be now deduced in a way similar  to 
the proof of Theorem~\ref{thm:reduction} for $e =\{x_0\}$, or, equivalently,  
by a direct argument using Theorem~\ref{thm:SpW},
Theorem~\ref{thm:equilocal} which allows to express the sum in Theorem~\ref{thm:SpW} 
in terms of the sum of minimum slopes $s_{0,n}^\nu$, $\nu\in \T_{x_0}(\Gamma)$,  
and Proposition~\ref{prop:coef} which provide an estimation for the sum of 
$s_{0,n}^\nu$ over $\nu \in \T_{x_0}(\Gamma)$. 
We omit the details.
\end{proof}

\end{document}